\documentclass[11pt,leqno,oneside]{amsart}
\usepackage{amsmath,amsfonts,amssymb,amsthm,enumerate,mathtools}
\usepackage{a4wide,color,framed,mathscinet}
\usepackage[a-2u]{pdfx}
\usepackage{hyperref}
\hypersetup{
	colorlinks,
	linkcolor=teal,
	urlcolor=violet,
	citecolor=violet}
\usepackage{doi}
\usepackage[a4paper, left=2cm, right=2cm, top=3cm, bottom=3cm]{geometry}
\usepackage[numbers,sort&compress]{natbib}
\usepackage[table]{xcolor}
\usepackage[textsize=tiny,textwidth=3.2cm,colorinlistoftodos]{todonotes}
\usepackage{enumitem}
\setlist[enumerate]{label={\rm(\roman*)},leftmargin=6ex}

\newcommand{\R}{\mathbb{R}}
\newcommand{\rn}{\mathbb{R}^n}
\newcommand{\N}{\mathbb{N}}
\newcommand{\MM}{\mathcal{M}}
\newcommand{\FF}{\mathcal{F}}
\newcommand{\RR}{\mathcal{R}}
\renewcommand{\d}{{\fam0 d}}
\newcommand{\kappab}{\kappa_\beta}
\newcommand{\med}{\operatorname{med}}
\newcommand{\mv}{\operatorname{mv}}
\newcommand{\mm}{\operatorname{m}}
\newcommand{\sgn}{\operatorname{sgn}}
\newcommand{\sprt}{\operatorname{sprt}}

\let\tilde\widetilde
\let\hat\widehat

\newtheoremstyle{MyPlain}{}{}{\itshape}{}{\bfseries}{.}{5pt plus 4pt minus 3pt}{\thmname{#1}\thmnumber{ #2}\thmnote{ \textbf{[#3]}}}
\theoremstyle{MyPlain}
\newtheorem{theorem}{Theorem}[section]
\newtheorem{lemma}[theorem]{Lemma}
\newtheorem{proposition}[theorem]{Proposition}

\newtheoremstyle{MyRemark}{}{}{\upshape}{}{\bfseries}{.}{5pt plus 1pt minus 1pt}{}
\theoremstyle{MyRemark}
\newtheorem{remark}[theorem]{Remark}

\numberwithin{equation}{section}

\expandafter\let\expandafter\oldproof\csname\string\proof\endcsname
\let\oldendproof\endproof
\renewenvironment{proof}[1][\proofname]{%
  \oldproof[{{\bf #1.}}]%
}{\oldendproof}

\makeatletter
\def\paragraph{\bigskip\@startsection{paragraph}{4}%
  \z@\z@{-\fontdimen2\font}%
  {\normalfont\bfseries}}
\makeatother

\makeatletter
\newcommand{\opnorm}{\@ifstar\@opnorms\@opnorm}
\newcommand{\@opnorms}[1]{%
	\left|\mkern-1.5mu\left|\mkern-1.5mu\left|
	#1
	\right|\mkern-1.5mu\right|\mkern-1.5mu\right|
}
\newcommand{\@opnorm}[2][]{%
  \mathopen{#1|\mkern-1.5mu#1|\mkern-1.5mu#1|}
  #2
  \mathclose{#1|\mkern-1.5mu#1|\mkern-1.5mu#1|}
}
\makeatother

\newcommand{\ib}{\frac{1}{\beta}}
\newcommand{\iib}{\frac{2}{\beta}}

\newcommand{\RG}{(\rn,\gamma_n)}
\newcommand{\WexpLb}{W^1\exp L^\beta\RG}
\newcommand{\dgn}{\d\gamma_n}

\newcommand{\normI}[1]{\opnorm*{\frac{1}{\,I\,}}_{{#1}\bigl(\Phi(t),\frac12\bigr)}}
\newcommand{\normIB}{\normI{L^{\tilde B}}}
\newcommand{\AL}{\underline{A}}
\newcommand{\aL}{\underline{a}}
\newcommand{\AU}{\mathcal A}

\begin{document}

\title{Moser inequalities in Gauss space}

\begin{abstract}
The sharp constants in a family of exponential Sobolev type inequalities in
Gauss space are exhibited. They constitute the Gaussian analogues of the Moser
inequality in the borderline case of the Sobolev embedding in the Euclidean
space. Interestingly, the Gaussian results have features in common with the
Euclidean ones, but also reveal marked diversities.
\end{abstract}

\author{Andrea Cianchi\textsuperscript{1}}
\address{\textsuperscript{1}Dipartimento di Matematica e Informatica ``Ulisse Dini'',
University of Florence,
Viale Morgagni 67/A, 50134
Firenze,
Italy}
\email{andrea.cianchi@unifi.it}
\urladdr{0000-0002-1198-8718}

\author{V\'\i t Musil\textsuperscript{1,2,3}}
\address{\textsuperscript{2}The Czech Academy of Sciences,
Institute of Mathematics,
\v Zitn\' a 25,
115~67, Prague~1,
Czech Republic}
\email{musil@math.cas.cz}
\urladdr{0000-0001-6083-227X}

\author{Lubo\v s Pick\textsuperscript{3}}
\address{\textsuperscript{3}Department of Mathematical Analysis,
Faculty of Mathematics and Physics,
Charles University,
So\-ko\-lo\-vsk\'a~83,
186~75 Praha~8,
Czech Republic}
\email{pick@karlin.mff.cuni.cz}
\urladdr{0000-0002-3584-1454}



\maketitle

\bibliographystyle{abbrv_doi}

\makeatletter
   \providecommand\@dotsep{2}
\makeatother

\section*{How to cite this paper}
\noindent
This paper has been accepted to \emph{Mathematische Annalen}
and the final publication is available at
\begin{center}
	\url{https://doi.org/10.1007/s00208-020-01956-z}.
\end{center}
Should you wish to cite this paper, the authors would like to cordially ask you
to cite it appropriately.

\section{Introduction and main results}

The present paper deals with a family of exponential type Sobolev inequalities
in Gauss space $\RG$, namely the space $\rn$ endowed with the Gauss probability
measure $\gamma_n$ given by
\begin{equation*}
	\dgn(x) = (2\pi)^{-\frac{n}{2}} e^{-\frac{|x|^2}{2}}\,\d x
		\quad\text{for $x\in \rn$.}
\end{equation*}
The inequalities to be considered admit diverse variants.  All of them concern,
for a given $\beta>0$, the uniform bound
\begin{equation} \label{april1}
		\int_{\rn} e^{\left(\kappa |u|
			\right)^{\frac{2\beta}{2+\beta}}} \,\dgn
		\le C
\end{equation}
 for suitable positive constants $\kappa$ and $C$, and  for every
weakly differentiable function $u$ in $\rn$ subject to a constraint on some
kind of exponential integrability for $|\nabla u|^\beta$, and  to the
normalization
\begin{equation}\label{mu}
	\mm(u)=0.
\end{equation}
Here, and in what follows, $\mm(u)$ denotes either the mean value $\mv (u)$ or
the median $\med (u)$ of $u$ over $\RG$.

The most straightforward version of the relevant gradient constraint reads
\begin{equation} \label{E:integral-form-M}
	\int_{\rn} {e^{|\nabla u|^\beta}}\,\dgn
		\le M
\end{equation}
for some constant $M>1$.   This assumption on $M$ is made since the integral in
\eqref{E:integral-form-M} cannot be smaller than $1$, and equals $1$ if and
only if $u$ is constant.

Inequalities of this form go back to \citep{Aid:94,Cia:09,Bob:99,Led:95}. They
can be equivalently stated as  embeddings of the Gaussian Orlicz-Sobolev spaces
$W^1\exp L^\beta\RG$ into the Orlicz spaces  $\exp
L^{\frac{2\beta}{2+\beta}}\RG$, associated with  Young functions equivalent
near infinity to $e^{t^\beta}$ and $e^{t^{\frac{2\beta}{2+\beta}}}$,
respectively.  In particular,   in \citep{Cia:09}  it is shown  that  the
exponent $\frac{2\beta}{2+\beta}$ in \eqref{april1} is the largest possible
that makes these embeddings true. Interestingly, since $\frac{2\beta}{2+\beta}
< \beta$, there is a loss in the degree of integrability between $|\nabla u|$
and $u$ in the exponential scale.  In fact, results of  \citep{Cia:09} ensure
that $\exp L^{\frac{2\beta}{2+\beta}}\RG$ is the optimal target space for
embeddings of $W^1\exp L^\beta\RG$ within the class of all Orlicz spaces on $\RG$,
and even in the larger class of all rearrangement-invariant spaces.

The exponential embeddings in question  extend, at a different scale,
a family of Gaussian embeddings for the Sobolev spaces $W^{1,p}\RG$, with $p
\in [1, \infty)$, into the Orlicz spaces $L^p(\log L)^{\frac p2}\RG$, where at least a minimal gain of integrability
between $|\nabla u|$
and $u$ is guaranteed.
This
result for $p=2$  was established in the seminal paper by Gross \citep{Gro:75},
whose researches  were also motivated by applications to quantum field theory
and to inequalities on infinite-dimensional  spaces.   The generalization to the case when $p\neq 2$
is contained in  \citep{Ada:79}.

In  \citep{Cia:09}, a borderline Gaussian Orlicz space, in the  region between
power and exponential type spaces, is exhibited with the property that
membership of $|\nabla u|$ in this space implies  that $u$ belongs exactly to
the same space,   this piece of information about its degree of integrability
being sharp. The Orlicz space in question is denoted by $\exp (\tfrac 14
\log^2L) \RG$, and is built upon any Young function equivalent to $e^{\frac 14
\log^2t}$ near infinity. The relevant Sobolev embedding thus  tells us that
$W^1 \exp (\tfrac 14 \log^2L) \RG$ is embedded into $\exp (\tfrac 14 \log^2L)
\RG$, and that the target space is optimal among all rearrangement-invariant
spaces.

Further results about Gaussian
Sobolev type inequalities are the subject of a rich literature in the areas of
convexity in high dimensions, isoperimetric inequalities, spectral theory,
probability, hypercontractive semigroups. Besides those mentioned above,
contributions in this connection include \citep{Bar:06, Bar:08,
Bob:98, Bra:07, Bob:97, Car:01, Cip:00, Fei:75, Fuj:11, Mil:09, Pel:93,
Rot:85}.

Our focus is on a sharp form of inequality \eqref{april1}. Specifically, we
investigate the optimal---largest possible---constant $\kappa$ for which
inequality \eqref{april1} holds under the normalization condition \eqref{mu},
and either  \eqref{E:integral-form-M} or some alternate closely related
assumption.

This can be regarded as a Gaussian counterpart of  the question addressed in
the celebrated paper by Moser \citep{Mos:70}, dealing with the optimal constant
in an exponential inequality established in \citep{Poh:65,Tru:67,Yud:61}. The
latter arises in the borderline case of the  Sobolev embedding theorem in the
Euclidean setting, namely in (subsets of) $\rn$ equipped with the Lebesgue
measure.  Moser's inequality tells us that there exists a constant
$C=C(n)$ such that
\begin{equation} \label{april3}
	\int_{\rn} \left(e^{\left(n\omega_n^{1/n} |u|\right)^{n'}} - 1\right)\d x
		\le C \lvert\sprt(u)\rvert
\end{equation}
for every  weakly differentiable  function $u$ in $\rn$, with support of finite
Lebesgue measure,  fulfilling
\begin{equation} \label{april4}
	\int_{\rn} |\nabla u|^n  \,\d x
		\le 1.
\end{equation}
Here, $\omega_n$ denotes the Lebesgue measure of the unit ball in $\rn$,
$\lvert\sprt(u)\rvert$ stands for the  measure of the support of $u$, and $n'=\tfrac
n{n-1}$. Moreover, the constant $n\omega_n^{1/n}$ is sharp in inequality
\eqref{april3}, since the integral on the left-hand side fails to be uniformly
bounded under constraint \eqref{april4} and under an upper bound for
$\lvert\sprt(u)\rvert$, if $n\omega_n^{1/n}$ is replaced by any larger
constant.

Such a result  has paved the way to  numerous  investigations on exponential
inequalities for limiting Sobolev embeddings, including versions for
higher-order derivatives \citep{Ada:88, Alb:08, Fon:93}, unrestricted  supports
\citep{Lam:13, Fon:18, Li:08, Ruf:05, Ish:11, Mas:15},  more general measures
in \eqref{april3} \citep{Cia:08,Fon:11}, subsets of $\rn$ and arbitrary
boundary values \citep{Cia:05,Fon:12,Lec:05}, Riemannian manifolds
\citep{Li:05,Fon:93,Kar:16,Bra:13,Yan:12,Bec:93}, the Heisenberg group
\citep{Coh:01} or more general Carnot groups \citep{Bal:03}, perturbations of
the space $W^{1,n}(\rn)$ \citep{Alv:96,Alb:08,Hen:03,Cer:10}.

The conclusions that will be derived on the Gaussian inequality \eqref{april1}
share some  traits with the Euclidean ones, but also exhibit sharp
dissimilarities. This is not only due to the presence of a measure that  decays
exponentially fast near infinity but also to an exponential integrand in the
gradient constraint.

Our results can be stated  with a gradient constraint either in integral form,
as in \eqref{E:integral-form-M}, or in a norm form. The two formulations are
not completely equivalent, because of the  nature of norms in Orlicz spaces.
Also, weak type norms of the gradient in exponential spaces---also called
Marcinkiewicz norms---are included in our discussion. In all these variants,
the sharp constant $\kappa$ in inequality \eqref{april1}, namely the supremum
among all values of $\kappa$ that render it true, turns out to depend only on
$\beta$ and agrees with
\begin{equation} \label{E:kappab}
	\kappab = \frac{1}{\sqrt{2}}+\frac{\sqrt{2}}{\beta}.
\end{equation}
Differences arise in connection with the central property that such a supremum
be attained or not, namely with the validity of inequality \eqref{april1} with
$\kappa = \kappab$.  Notice that the fact that $\kappab$ is independent of the dimension
$n$ is consistent with the whole theory of Gaussian  Sobolev inequalities.

Here, we state the result about problem
\eqref{april1}-\eqref{E:integral-form-M}.  The picture is completed   in
Section~\ref{main}, where variations on constraint \eqref{E:integral-form-M}
are analyzed.  In all cases, the conclusions take a different form depending on
whether $\beta\in(0,2]$ or $\beta\in(2,\infty)$.  The limiting situation when,
instead of  \eqref{E:integral-form-M}, a bound on $\|\nabla u\|_{L^\infty(\rn,
\gamma_n)}$ is imposed, is  considered as well.

\begin{theorem}[Integral form] \label{T:integral-form}
Let $n \ge 1$.
\par \noindent
{\rm Part 1}. Assume that $\beta\in(0,2]$.
\begin{enumerate}[label={\rm(1.\roman*)}]
	\item If $0<\kappa\le\kappab$, then for every $M>1$ there exists a constant
	$C=C(\beta,M)$ such that inequality \eqref{april1} holds for every  function
	$u$ obeying  \eqref{mu} and \eqref{E:integral-form-M}.

	\item If $\kappa>\kappab$, then for any $M>1$ there exists a function $u$
	obeying  \eqref{mu} and \eqref{E:integral-form-M} that makes the integral in
	\eqref{april1} diverge.
\end{enumerate}
{\rm Part 2}.  Assume that $\beta\in(2,\infty)$.
\begin{enumerate}[label={\rm(2.\roman*)}]
	\item If $0<\kappa<\kappab$, then for every $M>1$ there exists a constant
	$C=C(\beta,M)$ such that inequality \eqref{april1} holds for every  function
	$u$ obeying \eqref{mu} and \eqref{E:integral-form-M}.

	\item If $\kappa=\kappab$, then there exist $M>1$ and $C>0$ such that
	inequality \eqref{april1} holds for every function $u$ obeying
	\eqref{mu} and \eqref{E:integral-form-M}, and there exists $M> 1$ such that
	\eqref{april1} fails, whatever $C$ is, as $u$ ranges over all functions
	obeying \eqref{mu} and \eqref{E:integral-form-M}.

	\item If $\kappa>\kappab$, then for any $M>1$ there exists a function $u$
	obeying \eqref{mu} and \eqref{E:integral-form-M} that makes the integral in
	\eqref{april1} diverge.
\end{enumerate}
\end{theorem}

Let us briefly comment on some peculiarities of Theorem~\ref{T:integral-form}.
The appearance of a threshold value $\beta =2$, which dictates the form of the
result, is a new phenomenon in the frames of Moser  and Gaussian type
inequalities.  In particular, it is striking that the value of the constant $M$
appearing in condition \eqref{E:integral-form-M} is immaterial when $\beta \in
(0, 2]$, but affects the conclusions  if $\beta\in(2,\infty)$.  By contrast,
the value $1$ appearing on the right-hand side of \eqref{april4} is critical,
inasmuch as inequality \eqref{april3} fails if $1$ is replaced by any larger
constant.

One more unexpected assertion of Theorem~\ref{T:integral-form} is that,  if
$\kappa > \kappab$, then just  single functions $u$ can be exhibited, for which the
integral in \eqref{april1} diverges, to demonstrate the failure of inequality
\eqref{april1}.  Instead, the integral in \eqref{april3} is finite for each
function $u \in W^{1,n}(\rn)$ whose support has finite measure, even if $n
\omega_n^{1/n}$ is replaced by any larger constant.  Inequality  \eqref{april3}
fails in this case just because its left-hand side is not uniformly bounded by
some constant depending only on $\lvert\sprt(u)\rvert$.  An analogue in the Gaussian case
holds in the subspace $W^1\exp E^\beta\RG$ of those functions in $W^1\exp
L^\beta\RG$ such that
\begin{equation*}
	\int_{\rn} {e^{\lambda|\nabla u|^\beta}}\,\dgn
		< \infty
\end{equation*}
for every $\lambda >0$.

\begin{theorem}[Single functions in $W^1\exp E^\beta\RG$]
\label{T:E-space}
Let $\beta>0$ and let  $u\in W^1\exp E^\beta\RG$. Then
\begin{equation} \label{E:E-space-kappa}
	\int_{\rn} e^{(\kappa |u|)^\frac{2\beta}{2+\beta}}\,\dgn
		< \infty
\end{equation}
for every $\kappa >0$.
\end{theorem}

Like that of Moser's paper, and  those of most of the related contributions
mentioned above, our approach rests upon a suitable symmetrization argument,
which reduces the Gaussian inequalities in question to inequalities for
one-variable functions. The symmetrization of use in the present framework,
called  Ehrhard symmetrization in what follows, was introduced in
\cite{Ehr:84} and is in its turn related to the isoperimetric inequality in
Gauss space \cite{Bor:75, Sud:74}. Basic properties of Ehrhard  symmetrization
are recalled in Section~\ref{sec:sym}, where they are exploited in the proof of
some key inequalities for our method. The one-dimensional problems to be faced
after symmetrization present specific  difficulties compared with those arising
in the Euclidean setting. A distinctive complication is that both the
isoperimetric function in Gauss space and  its norms in exponential Orlicz
spaces do not admit  expressions in closed form. This calls for precise
asymptotic estimates for the relevant expressions, that are established in
Section \ref{SE:back}. Let us add that the choice of appropriate norms in the
Orlicz spaces is also critical   for certain inequalities to hold with exact
constants. With this material at disposal, our proofs of Theorems
\ref{T:integral-form} and \ref{T:E-space}, as well as those of the other main
results stated in Section \ref{main}, are accomplished in Section
\ref{SE:proofs}. The necessary function-space background is recalled in
Section \ref{sec:funct}  below.

\section{Function spaces}\label{sec:funct}

Basic definitions and properties concerning function spaces involved in our
discussion are collected in this section. For more details and proofs we refer
to the monographs  \citep{Ben:88} and \citep{Rao:91}.

Let $(\RR, \nu)$ be a probability space, namely a~measure space $\RR$ endowed
with a~probability measure~$\nu$. Assume  that $(\RR, \nu)$ is non-atomic. In
fact, we shall just be concerned with the case when $\RR$ is either  $\rn$
endowed with the Gauss measure ${\gamma_n}$, or $(0,1)$ endowed with the
Lebesgue measure. In the latter case, the measure will always be omitted in the
notation. More generally, we shall simply write $\RR$ instead of $(\RR, \nu)$
when no ambiguity can arise.  The notation $\MM(\RR)$ is employed for the space
of real-valued, $\nu$-measurable functions on $\RR$.

Let $\phi \in \MM(\RR)$.  The decreasing rearrangement  $\phi^*\colon[0,1]
\to [0,\infty]$ of $\phi$ is given by
\begin{equation*}
	\phi^*(s)
		= \inf\{t\ge0: \nu\left(\{x\in\RR : |\phi(x)|>t\}\right)\le s\}
		\quad\text{for $s\in [0,1]$.}
\end{equation*}
Similarly, the signed decreasing rearrangement $\phi^\circ\colon [0,1]
\to[-\infty, \infty]$ of $\phi$ is defined as
\begin{equation*}
	\phi^\circ(s)
		= \inf\{t\in\R: \nu(\{x\in\RR : \phi(x)>t\})\le s\}
		\quad\text{for $s\in[0,1]$.}
\end{equation*}
If $\phi$ is integrable on $\RR$, we also define the maximal function
$\phi^{**}\colon (0,1)\to[0,\infty]$ associated with $\phi^*$ as
\begin{equation*}
	\phi^{**}(s)
		= \frac{1}{s}\int_0^s \phi^*(r)\, \d r
		\quad\text{for $s\in (0,1)$.}
\end{equation*}
The functions $\phi^*$ and $\phi^{**}$ are non-increasing and $\phi^* \leq \phi^{**}$.

The Hardy-Littlewood inequality implies that, if $\phi, \psi \in\MM(\RR)$, then
\begin{equation}\label{HL}
	\int_E |\phi \,\psi |\,\d \nu
		\le \int_0^{|E|} \phi^*(s) \psi^*(s)\,\d s
\end{equation}
for every measurable set $E \subset \RR$.

The median $\med (\phi)$  and the mean value  $\mv (\phi)$ of $\phi$ are defined as
\begin{equation*}
	\med (\phi) = \phi^\circ (\tfrac 12)
		\quad\text{and}\quad
		\mv (\phi) = \int_\RR \phi \,\d\nu.
\end{equation*}
Of course, $\mv (\phi)$ is well defined only if  $\phi$  is integrable over $\RR$.

A Young function $A\colon[0, \infty )\to [0, \infty ]$ is a left-continuous
convex function vanishing at $0$, which is not constant in $(0, \infty)$. Any
Young function $A$ admits the representation
\begin{equation} \label{a}
	A(t) = \int_0^t a(\tau)\, \d \tau
	\quad\text{for $t \ge 0$,}
\end{equation}
for some non-decreasing left-continuous function $a\colon [0, \infty )\to [0,
\infty ]$.
Of course any finite-valued, nonnegative convex function on $[0,
\infty)$ vanishing at $0$ is continuous. Thus,  the additional assumption about
left-continuity of a Young function is only relevant in the case when it
attains the value $\infty$.

By $\widetilde A\colon[0,\infty)\to[0,\infty]$ we denote the Young conjugate of   $A$,
defined as
\begin{equation*} 
	\widetilde A(t)
		= \sup \{\tau t-A(\tau): \tau \ge 0\}
	\quad\text{for $t\ge 0$}.
\end{equation*}
The function $\widetilde A$ is also a Young function.  If $A$ is given by
\eqref{a}, then
\begin{equation*} 
	\widetilde A (t) = \int_0^ta^{-1}(\tau)\, \d\tau
		\quad\text{for $t \geq 0$,}
\end{equation*}
where $a^{-1}$ denotes the (generalized) left-continuous inverse of $a$.  The
very definition of Young conjugate ensures that
\begin{equation} \label{E:Young-ineq}
	\tau t \le A(\tau) + \tilde{A}(t)
		\quad\text{for $\tau,t>0$.}
\end{equation}
Moreover, equality holds in \eqref{E:Young-ineq} if either $\tau=a^{-1}(t)$ or
$t=a(\tau)$.

The Orlicz space $L^A(\RR)$ built upon a Young function $A$ is defined as
\begin{equation*} 
	L^A(\RR)
		= \bigg\{\phi\in \MM(\RR):
			\int_{\RR} A\biggl(\frac{|\phi |}{\lambda}\biggr)\, \d\nu  < \infty
			\,\,\text{for some $\lambda>0$}\bigg\}.
\end{equation*}
The space $L^A(\RR)$ is a Banach space equipped with the Luxemburg norm
given by
\begin{equation*}
	\|\phi\|_{L^A(\RR)}
		= \inf \bigg\{\lambda>0:
			\int_{\RR} A\bigg(\frac{|\phi |}{\lambda}\bigg)\,\d\nu\le 1
		\bigg\}
\end{equation*}
for $\phi \in L^A(\RR)$.  One has  that $L^A(\RR)=L^B(\RR)$ (up to equivalent
norms) if and only if $A$ and $B$ are Young functions equivalent near infinity,
in the sense that $A(c_1t) \le B(t) \le A(c_2t)$ for some constants $c_1$
and $c_2$, and for sufficiently large $t$.

Recall that
\begin{equation}\label{embeddings}
L^\infty (\RR) \to L^A(\RR) \to L^1(\RR)
\end{equation}
for every Young function $A$. Here, the arrow $\lq\lq \to "$ denotes continuous embedding.

The Orlicz norm
$\opnorm{\,\cdot \,}_{L^{A}(\RR)}$, given by
\begin{equation*}
	\opnorm{\phi}_{L^A(\RR)}
		= \sup\biggl\{ \int_{\RR} \phi\psi\,\d\nu:
				\int_{\RR} \tilde{A}(|\psi|)\,\d\nu\le 1
			\biggr\}
\end{equation*}
for $\phi \in L^A(\RR)$, is equivalent to the Luxemburg norm.

If $\phi\in L^{A}(\RR)$ and $E\subset\RR$ is a measurable set, we use the abridged
notations
\begin{equation*}
	\|\phi\|_{L^A(E)}
		= \|\phi\chi_E\|_{L^A(\RR)}
	 \quad\text{and}\quad
	 \opnorm{\phi}_{L^A(E)}
	 	= \opnorm{\phi\chi_E}_{L^A(\RR)}.
\end{equation*}
In particular,
\begin{equation} \label{E:Orl_char}
	\opnorm{1}_{L^A(E)}
		= \nu(E)\tilde{A}^{-1}\bigl(1/\nu(E)\bigr).
\end{equation}
Here, $\tilde{A}^{-1}$ denotes the (generalized) right-continuous inverse of
$\tilde A$.  A sharp form of the H\"older inequality in Orlicz spaces tells us
that
\begin{equation} \label{E:Holder}
	\int_{\RR} \phi\psi\,\d\nu
		\le \|\phi\|_{L^A(\RR)} \opnorm{\psi}_{L^{\tilde A}(\RR)}
\end{equation}
for every $\phi \in L^A(\RR)$ and $\psi \in L^{\tilde A}(\RR)$.

The Marcinkiewicz space $M^A(\RR)$, associated with a Young function $A$, is
defined as the space of all functions $\phi\in\MM(\RR)$ for which the norm
\begin{equation} \label{E:def-Marcinkiewicz-norm}
	\|\phi\|_{M^A(\RR)}
		= \sup_{s \in (0,1)}\frac{\phi^{**}(s)}{A^{-1}(1/s)}
\end{equation}
is finite.  We also define $m^A(\RR)$ as the collection of all functions
$\phi\in\MM(\RR)$ for which the quantity
\begin{equation} \label{E:def-Marcinkiewicz-quasi-norm}
	\|\phi\|_{m^A(\RR)}
		= \sup_{s \in (0,1)}\frac{\phi^{*}(s)}{A^{-1}(1/s)}
\end{equation}
is finite. Note that the functional $\|\cdot \|_{m^A(\RR)}$ is a quasi-norm, in
the sense that it enjoys the same properties of a norm, save that the triangle
inequality holds up to a multiplicative constant.  The embeddings $L^A(\RR) \to
M^A(\RR) \to m^A(\RR)$ hold for every Young function $A$, and
\begin{equation} \label{AtoM}
	\|\phi\|_{m^A(\RR)}
		\le \|\phi\|_{M^A(\RR)}
		\le \|\phi\|_{L^A(\RR)}
\end{equation}
for every $\phi \in L^A(\RR)$.

We denote by $E^A(\RR)$ the subspace of $L^A(\RR)$ defined by
\begin{equation} \label{Espace}
	E^A(\RR)
		= \bigg\{\phi\in \MM(\RR):
			\int_{\RR} A\biggl(\frac{|\phi |}{\lambda}\biggr)\, \d\nu  < \infty
			\,\,\text{for every $\lambda >0$}\bigg\}.
\end{equation}
The space $E^A(\RR)$  coincides with the  subspace of functions in $L^A(\RR)$
having an absolutely continuous norm. Recall that a function $\phi\in
L^A(\RR)$ is said to have an absolutely continuous norm if for every
$\varepsilon>0$ there exists $\delta>0$ such that
\begin{equation*}
	\|\phi \|_{L^A(G)} < \varepsilon
\end{equation*}
for every measurable set $G\subset\RR$ with $\nu(G)\le\delta$.

Let $\phi\in\MM(\RR)$. One has that $\phi\in L^A(\RR)$ if and only if
$\phi^* \in L^A(0,1)$.  Furthermore,
\begin{equation*} 
	\|\phi\|_{L^A(\RR)}
		= \|\phi^*\|_{L^A(0,1)}
		\quad\text{and}\quad
	\opnorm{\phi}_{L^A(\RR)}
		= \opnorm{\phi^*}_{L^A(0,1)}
\end{equation*}
for every $\phi \in \MM(\RR)$. Similarly, $\phi \in E^A(\RR)$ if and only if $\phi ^* \in E^A(0,1)$. Also,
$\phi \in M^A(\RR)$ if and only if $\phi^* \in M^A(0,1)$, and
\begin{equation*} 
	\|\phi\|_{M^A(\RR)}
		= \|\phi^*\|_{M^A(0,1)}
\end{equation*}
for every $\phi \in \MM(\RR)$. Obviously, one also has that
\begin{equation*} 
	\|\phi\|_{m^A(\RR)}
		= \|\phi^*\|_{m^A(0,1)}.
\end{equation*}
Hardy's lemma tells us that, given any nonnegative functions $\phi,\psi\in
\mathcal M(0,1)$ and any non-increasing function $\zeta\colon(0,1)\to[0,
\infty)$,
\begin{equation} \label{Hardylemma}
	\text{if\,  $\displaystyle \int _0^s \phi (r)\, \d r \le \int _0^s \psi (r)\, \d r$ \, for \, $s\in (0,1)$, \, then \, $\int _0^s \phi (r) \zeta (r)\, \d r \le \int _0^s \psi (r) \zeta (r)\, \d r$ \, for \, $s\in (0,1)$.}
\end{equation}
Given $\beta\in(0,\infty)$, we denote by $\exp L^\beta(\RR)$ the Orlicz space
associated with any Young function $B(t)$ equivalent to $e^{t^\beta}$ near
infinity. Its subspace $\exp E^\beta(\RR)$ is defined according to definition
\eqref{Espace}. Notice that, for this choice of $B$, one has that $L^B(\RR) =
M^B(\RR) = m^B(\RR)$, up to equivalent norms. By $L^p (\log L)^\alpha (\RR)$ we
denote the Orlicz space associated with any Young function $A(t)$ equivalent to
$t^p\log^\alpha t$ near infinity,  where  either  $p>1$ and $\alpha\in\R$, or
$p=1$ and $\alpha\ge 0$. The space $L^\infty(\RR)$ is also an Orlicz space
corresponding to the choice $A(t)=\infty \chi_{(1,\infty)}(t)$.

The Orlicz-Sobolev space $W^{1}L^A\RG$ associated with a Young function $A$ is
defined as
\begin{equation} \label{sobolev}
	W^{1}L^A\RG
		= \big\{u : \text{$u$ is weakly differentiable in $\rn$, and $|\nabla u|\in L^A\RG$}\big\}.
\end{equation}
Owing to the second embedding in \eqref{embeddings} and to the inclusion
$W^{1,1}\RG \subset L(\log L)^{\frac12}\RG$, any function $u\in W^1 L^A\RG$
belongs to $L^1\RG$. The space $W^{1}L^A\RG$, equipped with the norm given by
\begin{equation*}
	\|u\|_{W^{1}L^A\RG}
		= \|u\|_{L^1\RG}
		  + \|\nabla u\|_{L^A\RG}
\end{equation*}
for $u \in W^1 L^A\RG$, is a Banach space.

The space $W^1 E^A\RG$ is defined analogously, on replacing
$L^A\RG$ by $E^A\RG$ on the right-hand side of
equation \eqref{sobolev}.

The Orlicz-Sobolev spaces $W^1\exp L^\beta\RG$  and $W^1\exp E^\beta\RG$ are
associated with the spaces of exponential type $\exp L^\beta\RG$ and $\exp
E^\beta\RG$.

\section{Main results, continued}\label{main}

The results of this section complement Theorem~\ref{T:integral-form}, and
describe the conclusions that can be derived about the attainability of the
threshold constant $\kappab$ in inequality \eqref{april1} when a counterpart of
condition \eqref{E:integral-form-M} is prescribed in norm form.  Although the
norms in question are equivalent, inequality \eqref{april1} turns out to be
sensitive  to the  chosen norm, and hence the conclusions in its connection may
differ. The borderline case when the $L^\infty$ norm of the gradient replaces its norm in an exponential
space is also considered.

We begin by considering the case of the Luxemburg norm. Namely, we address  the
validity of inequality \eqref{april1} for functions $u$ fulfilling the condition
\begin{equation} \label{E:nabla-Orlicz}
	\|\nabla u\|_{L^B\RG}\le 1,
\end{equation}
where $B$ is any Young function such that
\begin{equation} \label{BN}
	B(t)=Ne^{t^\beta}
		\quad\text{for $t>t_0$,}
\end{equation}
for some $N>0$ and $t_0>0$.

For norms of this type, the situation is analogous to that stated in
Theorem~\ref{T:integral-form} under a constraint in integral form.

\begin{theorem}[Luxemburg norms]
\label{T:norm-form}
Let $n \ge 1$.
\\{\rm Part 1}. Assume that $\beta\in (0,2]$.
\begin{enumerate}[label={\rm(1.\roman*)}]
	\item If $0<\kappa\le\kappab$, then for every $N>0$ and for every Young
	function $B$ as in \eqref{BN}, there exists a constant $C=C(\beta,N,t_0)$
	such that inequality \eqref{april1} holds for every function $u$ obeying
	\eqref{mu} and \eqref{E:nabla-Orlicz}.

	\item If $\kappa >\kappab$, then for any $N>0$
	and for any Young function $B$ as in \eqref{BN},
	there exists a function $u$
	obeying  \eqref{mu} and \eqref{E:nabla-Orlicz} that makes the integral in
	\eqref{april1} diverge.
\end{enumerate}
{\rm Part 2}.  Assume that $\beta\in(2,\infty)$.
\begin{enumerate}[label={\rm(2.\roman*)}]
	\item If $0<\kappa<\kappab$, then for every $N>0$, and for every Young
	function $B$ as in \eqref{BN} there exists a constant $C=C(\beta,N,t_0)$ such
	that inequality \eqref{april1} holds for every function $u$ obeying
	\eqref{mu} and \eqref{E:nabla-Orlicz}.

	\item If $\kappa=\kappab$,  then for every $N>0$, there exist a Young
	function $B$ as in \eqref{BN} and  a constant $C>0$ such that inequality
	\eqref{april1} holds for every  function $u$ obeying   \eqref{mu} and
	\eqref{E:nabla-Orlicz}, and there exists a Young function $B$ as in
	\eqref{BN} such that inequality \eqref{april1} fails, whatever $C$ is, as $u$
	ranges over all functions obeying  \eqref{mu} and \eqref{E:nabla-Orlicz}.

	\item If $\kappa>\kappab$, then for any $N>0$
	and for any Young function $B$ as in \eqref{BN},
	there exists a function $u$ obeying \eqref{mu} and
	\eqref{E:nabla-Orlicz} that makes the integral in \eqref{april1} diverge.
\end{enumerate}
\end{theorem}

Let us next examine constraints on trial functions in \eqref{april1} imposed in
terms of an exponential Marcinkiewicz norm defined as in
\eqref{E:def-Marcinkiewicz-norm}, or a quasi-norm given as in
\eqref{E:def-Marcinkiewicz-quasi-norm}. Specifically, we take into account
functions $u$ subject to \eqref{mu} and either condition
\begin{equation} \label{E:nabla-Marcinkiewicz-M}
	\|\nabla u\|_{M^B\RG}\le 1,
\end{equation}
or
\begin{equation} \label{E:nabla-Marcinkiewicz-m}
	\|\nabla u\|_{m^B\RG}\le 1,
\end{equation}
where $B$ is as in \eqref{BN}. Interestingly, the result differs from that of Theorem~\ref{T:norm-form}, but is the same in both
cases \eqref{E:nabla-Marcinkiewicz-M} and \eqref{E:nabla-Marcinkiewicz-m}.

\begin{theorem}[Marcinkiewicz norms] \label{T:weak-form}
Let $n \ge 1$.
\\\emph{Part 1}.  Assume that $\beta\in(0,2]$.
\begin{enumerate}[label={\rm(1.\roman*)}]
	\item If $0<\kappa <\kappab$, then for every $N>0$ and for every Young
	function $B$ as in \eqref{BN}, there exists a constant
	$C=C(\beta,\kappa,N,t_0)$ such that inequality \eqref{april1} holds for every
	function $u$ obeying \eqref{mu} and \eqref{E:nabla-Marcinkiewicz-M}.

	\item If $\kappa\ge \kappab$, then for any $N>0$
	and for any Young function $B$ as in \eqref{BN},
	there exists a function
	$u$ obeying \eqref{mu} and \eqref{E:nabla-Marcinkiewicz-M} that makes the
	integral in \eqref{april1} diverge.
\end{enumerate}
\emph{Part 2}.  Assume that $\beta\in(2,\infty)$.
\begin{enumerate}[label={\rm(2.\roman*)}]
	\item If $0<\kappa<\kappab$, then for every $N>0$ and for every Young
	function $B$ as in \eqref{BN}, there exists a constant $C=C(\beta,\kappa,N,
	t_0)$ such that inequality \eqref{april1} holds for every function $u$
	obeying \eqref{mu} and \eqref{E:nabla-Marcinkiewicz-M}.

	\item If $\kappa=\kappab$, then for every $N>0$, there exist a Young function
	$B$ as in \eqref{BN} and  a constant $C>0$ such that inequality
	\eqref{april1} holds for every function $u$ obeying \eqref{mu} and
	\eqref{E:nabla-Marcinkiewicz-M}, and there exists a Young function $B$ as in
	\eqref{BN} such that inequality \eqref{april1} fails, whatever $C$ is, as $u$
	ranges over all functions obeying \eqref{mu} and
	\eqref{E:nabla-Marcinkiewicz-M}.

	\item If $\kappa>\kappab$, then for  any $N>0$
	and for any Young function $B$ as in \eqref{BN},
	there exists a function $u$ obeying \eqref{mu} and
	\eqref{E:nabla-Marcinkiewicz-M} that makes the integral in \eqref{april1}
	diverge.
\end{enumerate}
The same statement holds if  condition \eqref{E:nabla-Marcinkiewicz-M} is
replaced by \eqref{E:nabla-Marcinkiewicz-m} throughout.
\end{theorem}

\begin{remark}
Theorem \ref{T:weak-form} shows one more diversity between  Gaussian and
Euclidean Moser type inequalities. Indeed, part (2.ii) tells us that the
threshold value $\kappa _\beta$ is admissible in inequality \eqref{april1}, at
least if $\beta >2$, under the gradient constraint of Marcinkiewicz type
\eqref{E:nabla-Marcinkiewicz-M} or \eqref{E:nabla-Marcinkiewicz-m}, for
suitable Young functions $B$ fulfilling condition \eqref{BN}. This is never the
case in the corresponding Euclidean results when Marcinkiewicz type norms of
the gradient are employed \citep{Alv:96,Alb:08}.
\end{remark}

\begin{remark}
In view of the inequalities in \eqref{AtoM}, the conclusions in the negative
direction contained in Theorem~\ref{T:norm-form}
imply those of Theorem~\ref{T:weak-form} about the
norm $\|\cdot\|_{M^B\RG}$, and the latter imply those about
 the quasi-norm $\|\cdot\|_{m^B\RG}$. Of course, reverse implications   hold
about the conclusions in the positive direction.
\end{remark}

\begin{remark}
Condition \eqref{BN} can be relaxed by requiring that there exist constants
$N_2>N_1>0$ and $t_0>0$ such that
\begin{equation} \label{E:weakerhp}
	N_1e^{t^\beta} \le B(t) \le N_2e^{t^\beta}
		\quad\text{for $t > t_0$.}
\end{equation}
Properly modified statements of Theorems \ref{T:norm-form} and
\ref{T:weak-form} hold under assumption \eqref{E:weakerhp}, with $N$ replaced by
$N_1$ in the assertions in the positive direction, and by $N_2$ in those in the
negative direction.
\end{remark}

Our last main result deals with a limiting version, as $\beta \to \infty$, of
inequality \eqref{april1} for functions subject to condition
\eqref{E:nabla-Orlicz}. The resulting inequality is
\begin{equation} \label{aprilinf}
		\int_{\rn} e^{\left(\kappa |u|
			\right)^{2}} \,\dgn
		\le C\,,
\end{equation}
under  conditions \eqref{mu} and
\begin{equation} \label{Linfinity}
	\|\nabla u\|_{L^\infty\RG}
		\le 1.
\end{equation}
The exponent $2$  is the largest admissible for $|u|$
in \eqref{aprilinf} under assumption \eqref{Linfinity}.  Also, the threshold
value of $\kappa$ in \eqref{aprilinf} is $\frac{1}{\sqrt{2}}$, namely
$\lim_{\beta\to\infty} \kappab$.

\begin{theorem}[$L^\infty$ norm]
\label{T:L-infty}
Let $n \ge 1$.
\begin{enumerate}
	\item If $0<\kappa<\frac{1}{\sqrt{2}}$, then there exists a constant
	$C=C(\kappa)$ such that inequality \eqref{aprilinf} holds for every function
	$u$ obeying \eqref{mu} and \eqref{Linfinity}

	\item If $\kappa\ge\frac{1}{\sqrt{2}}$, then there exists a function  $u$
	obeying \eqref{mu}  and \eqref{Linfinity}, that makes the integral in
	\eqref{aprilinf} diverge.
\end{enumerate}
\end{theorem}

\section{Ehrhard  symmetrization and ensuing inequalities}\label{sec:sym}

Key tools in our approach are some rearrangement inequalities for the gradient
of Sobolev functions on Gauss space. These inequalities in their turn rely upon
the isoperimetric inequality that links the Gauss measure of a set $E\subset
\rn$ to its Gauss perimeter. Recall that the Gauss perimeter
$P_{\gamma_n}(E)$ of a measurable set $E$ can be defined as
\begin{equation*}
	P_{\gamma_n}(E) = \frac 1{(2\pi)^{\frac n2}}
		\int_{\partial^M E}e^{-\frac{|x|^2}2}\,\d \mathcal H^{n-1}(x),
\end{equation*}
where $\mathcal H^{n-1}$ denotes the $(n-1)$-dimensional Hausdorff
measure, and $\partial^M E$ the essential boundary of $E$ in the sense of
geometric measure theory, namely the set of points of $\rn$ at which the
density of $E$ is neither $0$ nor $1$.
The Gaussian isoperimetric inequality
asserts that half-spaces minimize Gauss perimeter among all measurable
subsets of $\rn$ with prescribed Gauss measure \citep{Bor:75, Sud:74}.  Note
that
\begin{equation*}
	 \gamma_n(\{x\in\rn: x_1\ge t\})= \Phi(t)
		\quad\text{for $t\in\R$,}
\end{equation*}
where $\Phi\colon\R\to(0,1)$ is the function defined as
\begin{equation} \label{E:Phi-def}
	\Phi(t) = \frac{1}{\sqrt{2\pi}} \int_{t}^{\infty} e^{-\frac{\tau^2}{2}}\,\d\tau
		\quad\text{for $t\in\R$.}
\end{equation}
Moreover,
\begin{equation*}
	P_{\gamma_n} (\{x\in\rn: x_1\ge t\})
		=  \frac{1}{\sqrt{2\pi}} e^{-\frac{t^2}{2}}
		\quad\text{for $t\in\R$.}
\end{equation*}
Here, $x_1$ denotes the first component of the point $x\in \rn$.
Thereby, on defining the function $I\colon [0,1]\to [0,\infty)$ as
\begin{equation*} 
	I(s) = \frac{1}{\sqrt{2\pi}} e^{-\frac{\Phi^{-1}(s)^2}{2}}
		\quad\text{ for $s\in(0,1)$,}
\end{equation*}
and $I(0)=I(1)=0$,
the Gaussian isoperimetric inequality takes the analytic form
\begin{equation*}
	I(\gamma _n(E)) \le P_{\gamma_n}(E)
\end{equation*}
for every measurable set $E\subset\rn$. The function $I$ is accordingly called
the isoperimetric function (or isoperimetric profile) of Gauss space. Note that
it is symmetric about $\tfrac 12$, namely
\begin{equation} \label{symmetry}
	I(s) = I(1-s)
		\quad\text{for $s\in [0,1]$.}
\end{equation}
Also,
\begin{equation} \label{june30}
 - \Phi'(t) = 	I(\Phi(t))
		\quad\text{for $t \in \mathbb R$.}
\end{equation}

An Ehrhard  symmetral of a function $u\in\MM\RG$ is a function,
equimeasurable with $u$, whose level sets are half-spaces. Thus, the function
$u^\bullet\colon\rn\to\R$ defined as
\begin{equation*}
	u^\bullet (x) = u^\circ (\Phi (x_1))
		\quad\text{for $x\in\rn$,}
\end{equation*}
is an Ehrhard  symmetral of $u$.

The following result is established in  \citep[Lemma~3.3]{Cia:09}, and is the
point of departure in the proof  of fundamental properties of $u^\bullet$.

\begin{proposition} \label{P:L3.3}
Assume that  $u \in W^{1,1}\RG$.
Then  the function $u^\circ$ is locally
absolutely continuous in $(0,1)$, the function $u^\bullet \in W^{1,1}\RG$, and
\begin{equation} \label{PS}
	\int_{0}^s (-u^\circ{}'I)^*(r)\,\d r
		= \int_{0}^s |\nabla u^\bullet|^*(r)\,\d r
		\le \int_{0}^s |\nabla u|^*(r)\,\d r
		\quad\text{for $s\in[0,1]$}.
\end{equation}
\end{proposition}

A Gaussian P\'olya-Szeg\H o principle on the non-increase of Lebesgue
\cite{Ehr:84}, and more generally Orlicz \cite{Cia:09}, gradient norms under
Ehrhard  symmetrization, can immediately be derived from
Proposition~\ref{P:L3.3}, via Hardy's lemma \eqref{Hardylemma}.

\begin{proposition} \label{P:PSnorm}
Let $A$ be a Young function.  Assume that  $u \in W^{1,A}\RG$.
Then $u^\bullet \in W^{1,A}\RG$, and
\begin{equation}\label{PSnorm1}
	\lVert-u^\circ{}'I  \rVert_{L^A(0,1)}= \|\nabla u^\bullet\|_{L^A\RG}
		\le \|\nabla u\|_{L^A\RG}.
\end{equation}
\end{proposition}

The next Proposition can serve as a replacement for the P\'olya-Szeg\H o
inequality \eqref{PSnorm1} in dealing with certain functionals that depend on
the gradient, but are not norms.

\begin{proposition} \label{P:symubound}
Assume that the function  $u \in W^{1,1}\RG$ satisfies $\med (u)=0$.
Then
\begin{equation} \label{eq:symubound1}
	0\le u^\circ(s) \le
		\frac{1}{I(s)}\int_{0}^{s} |\nabla u|^*(r)\,\d r
		+ \int_{s}^{\frac{1}{2}} \frac{|\nabla u|^*(r)}{I(r)}\,\d r
		\quad\text{for $s\in(0,\tfrac12]$}
\end{equation}
and
\begin{equation} \label{eq:symubound2}
	0\le -u^\circ(1-s) \le
		\frac{1}{I(s)}\int_{0}^{s} |\nabla u|^*(r)\,\d r
		+ \int_{s}^{\frac{1}{2}} \frac{|\nabla u|^*(r)}{I(r)}\,\d r
		\quad\text{for $s\in(0,\tfrac12]$}.
\end{equation}
\end{proposition}

\begin{proof}
Proposition~\ref{P:L3.3}, combined
with Hardy's lemma \eqref{Hardylemma}, implies that
\begin{equation} \label{eq:aug1}
	\int_{0}^s (-u^\circ{}'I)^*(r)\zeta(r)\,\d r
		\le \int_{0}^s |\nabla u|^*(r)\zeta(r)\,\d r
		\quad\text{for $s\in(0,1)$,}
\end{equation}
for any non-increasing function $\zeta\colon(0,1)\to[0,\infty)$.  Since we are
assuming that $\med (u)= u^\circ(\tfrac12)=0$, we have that $u^\circ(s)\ge 0$
for $s\in(0,\tfrac12]$, and
\begin{align}\label{july1}
	\begin{split}
	u^\circ(s)
		& = \int_{s}^{\frac12} -u^\circ{}'(r)\,\d r
		 = \int_{0}^1 \frac{\chi_{(s,\frac12)}(r)}{I(r)}
			\left( -u^\circ{}'(r)I(r) \right)\,\d r
			\\
		& \le \int_{0}^1 \biggl(\frac{\chi_{(s,\frac12)}}{I}\biggr)^*\!(r)
			\left( -u^\circ{}'I \right)^*\!(r)\,\d r
		 \le \int_{0}^1 \biggl(\frac{\chi_{(s,\frac12)}}{I}\biggr)^*\!(r)
			|\nabla u|^*(r)\,\d r
			\quad\text{for $s\in(0,\tfrac12]$,}
	\end{split}
\end{align}
where the first inequality follows from Hardy-Littlewood inequality \eqref{HL}
and the second one is due to \eqref{eq:aug1}.  Furthermore, inasmuch as $I$ is
increasing in $(0,\tfrac12]$, if $s\in(0,\tfrac12]$, then
\begin{equation}\label{july2}
	\biggl(\frac{\chi_{(s,\frac12)}}{I}\biggr)^*\!(r)
		=\frac{\chi_{(0,\frac12-s)}(r)}{I(r+s)}
		\quad\text{for $r\in [0, 1]$.}
\end{equation}
On the right-hand side of equality \eqref{july2}, and in similar equalities
below, there is a slight abuse of notation, since $I$ is only defined in $[0,
1]$.
However, this is immaterial, since $\chi_{(0, \frac 12-s)}(r)=0$ if $r+s>\frac 12$.
From equations \eqref{july1} and \eqref{july2}, one deduces that
\begin{equation} \label{eq:aug2}
	0\le u^\circ(s) \le \int_{0}^{\frac12-s} \frac{|\nabla u|^*(r)}{I(r+s)}\,\d r
		\quad\text{for $s\in(0,\tfrac12]$}.
\end{equation}
Similarly, $u^\circ (s) \leq 0$  for $s \in [\tfrac12,1)$ and
\begin{equation*}
	-u^\circ(1-s)
		= \int_{\frac12}^{1-s} -u^\circ{}'(r)\,\d r
		\le \int_{0}^1 \biggl(\frac{\chi_{(\frac12,1-s)}}{I}\biggr)^*\!(r)
			|\nabla u|^*(r)\,\d r
			\quad\text{for $s\in(0,\tfrac12]$.}
\end{equation*}
Also, owing to  equation \eqref{symmetry} and to the monotonicity of $I$ on
$(0,\tfrac12]$, if $s\in(0,\tfrac12]$, then
\begin{equation*}
	\biggl(\frac{\chi_{(\frac12,1-s)}}{I}\biggr)^*\!(r)
		= \biggl(\frac{\chi_{(s,\frac12)}}{I}\biggr)^*\!(r)
		= \frac{\chi_{(0,\frac12-s)}(r)}{I(r+s)}
		\quad\text{for $r\in[0, 1]$.}
\end{equation*}
Hence,
\begin{equation} \label{eq:aug3}
	0\le -u^\circ(1-s) \le \int_{0}^{\frac12-s} \frac{|\nabla u|^*(r)}{I(r+s)}\,\d r
		\quad\text{for $s\in(0,\tfrac12]$}.
\end{equation}
Now, define the function $\hat I\colon[0,1]\to[0,\infty)$ as
\begin{equation*}
	\hat I(s)
		= \begin{cases}
			I(s) & \text{for $s\in[0,\frac12]$}
				\\
			I(\frac12) & \text{for $s\in(\frac12,1]$}.
		\end{cases}
\end{equation*}
Then,
\begin{align} \label{eq:aug4}
	\begin{split}
	\int_{0}^{\frac12-s} \frac{|\nabla u|^*(r)}{I(r+s)}\,\d r
		& \le \int_{0}^{\frac12} \frac{|\nabla u|^*(r)}{\hat I(r+s)}\,\d r
			 = \int_{0}^{s} \frac{|\nabla u|^*(r)}{\hat I(r+s)}\,\d r
			+ \int_{s}^{\frac12} \frac{|\nabla u|^*(r)}{\hat I(r+s)}\,\d r
			\\
		& \le \frac{1}{\hat I(s)} \int_{0}^{s} |\nabla u|^*(r)\,\d r
			+ \int_{s}^{\frac12} \frac{|\nabla u|^*(r)}{\hat I(r)}\,\d r
			\\
		& \le \frac{1}{I(s)} \int_{0}^{s} |\nabla u|^*(r)\,\d r
			+ \int_{s}^{\frac12} \frac{|\nabla u|^*(r)}{I(r)}\,\d r
			\quad\text{for $s\in(0,\tfrac12]$,}
	\end{split}
\end{align}
where the first and the last inequalities hold since $I=\hat I$
on $(0,\tfrac12]$, and the second is due to the monotonicity of $\hat I$.
Inequality \eqref{eq:symubound1} now follows from \eqref{eq:aug2} and
\eqref{eq:aug4}, and inequality \eqref{eq:symubound2} from \eqref{eq:aug3} and
\eqref{eq:aug4}.
\end{proof}

A sharp estimate for the difference between the median and the mean value of
any Sobolev function in terms of the $L^1\RG$ norm of its gradient is the
subject of the following proposition.

\begin{proposition} \label{P:med-mv}
Let $u \in W^{1,1}\RG$. Then
\begin{equation} \label{E:med-mv}
	|\med (u) - \mv (u)|
		\le \sqrt{\frac{\pi}{2}}\, \|\nabla u\|_{L^1\RG}.
\end{equation}
The constant $\sqrt{\frac{\pi}{2}}$ in inequality \eqref{E:med-mv} is sharp.
\end{proposition}

\begin{proof}
Owing to Proposition~\ref{P:L3.3},  the function $u^\circ$ is locally absolutely
continuous in $(0,1)$. Hence, by Fubini's theorem,
\begin{align} \label{jun1}
	\begin{split}
	u^\circ(s) - \mv (u)
		& = u^\circ(s) - \int_{0}^1 u^\circ(r)\,\d r
			= \int_{0}^1 \bigl( u^\circ(s) - u^\circ(r)\bigr)\,\d r
			= \int_{0}^1 \int_r^s u^\circ{}'(\varrho)\,\d \varrho\,\d r
			 \\
		&  = \int_{0}^s r u^\circ{}'(r)\,\d r
				- \int_{s}^1 (1-r)u^\circ{}'(r)\,\d r
		 = \int_{0}^1 \bigl( \chi_{(s,1)}(r) - r\bigr)\bigl(-u^\circ{}'(r)\bigr)\,\d r
	 \end{split}
\end{align}
for $s\in(0,1)$. Therefore
\begin{align} \label{E:med-mv-1}
	\begin{split}
	|\med (u) - \mv (u)|
		 = | u^\circ(\tfrac12) - \mv (u)|
		& = \biggl| \int_{0}^1 \frac{\chi_{(\frac12,1)}(s)-s}{I(s)}
					\bigl(-u^\circ{}'(s)I(s)\bigr)\,\d s \biggr|
			\\
		& \le \sup_{s\in(0,1)} \frac{|\chi_{(\frac12,1)}(s)-s|}{I(s)}
					\int_{0}^1 \bigl(-u^\circ{}'(s)I(s)\bigr)\,\d s.
	\end{split}
\end{align}
Now, notice that the function $s\mapsto s/I(s)$ is increasing on $(0, \tfrac
12)$.  Indeed, this  is equivalent to the fact that the function $t \mapsto
\Phi (t) e^{\frac {t^2}2}$ is decreasing on $(0,\infty)$, a property that can
be easily verified via differentiation and by the inequality
\begin{equation*}
	\int_t^\infty e^{-\frac{\tau^2}2}\,\d\tau
		\le \frac{e^{-\frac {t^2}2}}t
			\quad\text{for $t>0$,}
\end{equation*}
which is shown e.g.\ in \citep[Lemma~3.4]{Cia:11}.  From the monotonicity of
$s/I(s)$ in $(0, \tfrac 12)$ and property \eqref{symmetry} we have that
\begin{align}
\label{E:med-mv-2}
	\sup_{s\in(0,1)} \frac{|\chi_{(\frac12,1)}(s)-s|}{I(s)}
					 = \max\biggl\{
					\sup_{s\in(0,\frac12)} \frac{s}{I(s)},
					\sup_{s\in(\frac12,1)} \frac{1-s}{I(1-s)}
				\biggr\}
		 = \sup_{s\in(0,\frac12)} \frac{s}{I(s)}
			= \frac{1}{2I(\frac12)} = \sqrt{\frac{\pi}{2}}.
\end{align}
Furthermore, by Proposition~\ref{P:L3.3},
\begin{equation}
\label{E:med-mv-3}
	\int_{0}^1 \bigl(-u^\circ{}'(s)I(s)\bigr)\,\d s
	 = \int_{0}^1 (-u^\circ{}'I)^*(s)\,\d s
	 \le \int_{0}^1 |\nabla u|^*(s)\,\d s
	 = \|\nabla u\|_{L^1\RG}.
\end{equation}
On combining estimates \eqref{E:med-mv-1}, \eqref{E:med-mv-2} and
\eqref{E:med-mv-3}, one obtains \eqref{E:med-mv}.

The fact that the constant $\sqrt{\pi/2}$ in inequality \eqref{E:med-mv} is the
smallest possible can be verified on testing  the inequality on the sequence
$\{u_k\}$ defined as
\begin{equation*}
	u_k (x)
	= \begin{cases}
	0
	& \quad \text{for $x_1 \in(-\infty,0]$}
		\\
	k x_1
	& \quad \text{for $x_1 \in(0,\tfrac 1k]$}
		\\
	1
	& \quad \text{for $x_1 \in (\tfrac 1k,\infty)$.}
	\end{cases}
\end{equation*}
Indeed, $\med (u_k) = 0$ for $k\in\N$, $\lim_{k\to\infty} \mv(u_k) = \frac 12$,
and $\lim_{k\to\infty} \|\nabla u_k\|_{L^1\RG} = \frac 1{\sqrt{2\pi}}$.
\end{proof}

Given a Young function $B$, define the functions $\FF_{L^B}$ and $\FF_{m^B}$
from $(0, \infty)$ into $(0, \infty]$ as
\begin{equation} \label{E:F-LB}
	\FF_{L^B}(t)
		= \normIB
		+ \sqrt{\frac{\pi}{2}} B^{-1}(1)
\end{equation}
and
\begin{equation} \label{E:F-mB}
	\FF_{m^B}(t) =
		e^{\frac{t^2}{2}} \int_{t}^{\infty}
			B^{-1}\left( \frac{1}{\Phi(\tau)} \right) e^{-\frac{\tau^2}{2}}\,\d \tau
		+ \int_{0}^{t} B^{-1}\left( \frac{1}{\Phi(\tau)} \right)\,\d \tau
		+ \sqrt{\frac{\pi}{2}}
				\int_{0}^{1}B^{-1}\left( \frac{1}{s} \right)\d s
\end{equation}
for $t >0$.

The next lemma provides us with a  bound for the integral in \eqref{april1}
for any function $u$ satisfying either condition \eqref{E:nabla-Orlicz} or
\eqref{E:nabla-Marcinkiewicz-m}. Such a bound amounts to an integral depending
on either the function $\FF_{L^B}$ or $\FF_{m^B}$, respectively.

\begin{lemma} \label{L:Holder-estimates}
Let $\beta>0$ and $\kappa>0$, and let $B$ be a Young function.  Let  $X$
denote either  $L^B$ or $m^B$,  and let $\FF_X$ be defined as in \eqref{E:F-LB}
or \eqref{E:F-mB}, respectively. Then
\begin{equation} \label{E:exponential-estimate}
	\int_{\rn}
		e^{\left(\kappa |u| \right)^{\frac{2\beta}{2+\beta}}} \dgn
		\le \sqrt{\frac{2}{\pi}}
			\int_{0}^{\infty} e^{\left[ \kappa\FF_X(t) \right]^{\frac{2\beta}{2+\beta}}
				-\frac{t^2}{2}}\,\d t
\end{equation}
for every weakly differentiable function $u$ in $\rn$ satisfying \eqref{mu} and
such that
\begin{equation} \label{E:gradient-X-1}
	\|\nabla u\|_{X\RG}\le 1\,.
\end{equation}
\end{lemma}

\begin{proof}
Assume that $u$ obeys $\mv (u)=0$, the case when $\med (u)=0$ being even simpler.
Let us set $v=u-\med (u)$. Then $v$ is weakly differentiable, $\med(v)=0$
and $\nabla u =\nabla v$.

Let us begin by considering the case when  $X=L^B$.
By H\"older's inequality \eqref{E:Holder}, Proposition~\ref{P:PSnorm}
and \eqref{E:gradient-X-1}, we have that
\begin{align} \label{E:v-circ-bound-L-first}
	\begin{split}
	0 \leq v^\circ(s)
		& = \int_{s}^{\frac12} -v^\circ{}'(r)I(r)\frac{\,\d r}{I(r)}
			\le \lVert-v^\circ{}'I\rVert_{L^B(0,\frac12)}
			\opnorm*{\frac{1}{I}}_{L^{\tilde B}(s,\frac12)}
				\\
		& \le \|\nabla v\|_{L^B\RG}
			\opnorm*{\frac{1}{I}}_{L^{\tilde B}(s,\frac12)}
			\leq \opnorm*{\frac{1}{I}}_{L^{\tilde B}(s,\frac12)}
	\end{split}
\end{align}
for any $s\in(0,\tfrac12]$. On the other hand, owing to equation  \eqref{symmetry},
\begin{align}\begin{split} \label{E:v-circ-bound-L-second}
	0& \leq  -v^\circ(s)
		  = \int_{\frac12}^{s} -v^\circ{}'(r)I(r)\frac{\,\d r}{I(r)}
			\\ &  \le \|\nabla v\|_{L^B\RG}
					\opnorm*{\frac{1}{I}}_{L^{\tilde B}(\frac12,s)}
		 \leq \opnorm*{\frac{1}{I}}_{L^{\tilde B}(1-s,\frac12)}\quad \text{for $s\in(\tfrac12,1)$.}
\end{split}
\end{align}
Next,  from Proposition~\ref{P:med-mv}, H\"older's inequality \eqref{E:Holder}
and equation \eqref{E:Orl_char} we infer that
\begin{equation*}
	|\med (u)|
		\le \sqrt{\frac{\pi}{2}} \|\nabla u\|_{L^1\RG}
		\le \sqrt{\frac{\pi}{2}} \|\nabla u\|_{L^B\RG}
			\opnorm*{1}_{L^{\tilde B}\RG}
		\le \sqrt{\frac{\pi}{2}} B^{-1}(1),
\end{equation*}
whence, on setting $C= \sqrt{\frac{\pi}{2}} B^{-1}(1)$,
\begin{equation} \label{E:uv-diff}
	|u(x)|
		\le |v(x)| +  |\med (u)|
		\le |v(x)| + C
		\quad\text{for $x\in\rn$}.
\end{equation}
By \eqref{E:uv-diff}, since $\med(v)=0$,
\begin{align} \label{E:exponential-estimate-u}
	\begin{split}
	\int_{\rn} e^{\left(\kappa |u| \right)^{\frac{2\beta}{2+\beta}}}\dgn
 		& \leq \int_{\rn}
				e^{\left(\kappa |v| + \kappa C \right)^{\frac{2\beta}{2+\beta}}}\dgn
			\\
		& = \int_{0}^{\frac{1}{2}}
				e^{\left(\kappa v^\circ(s) + \kappa C \right)^{\frac{2\beta}{2+\beta}}}\d s
			+ \int_{0}^{\frac{1}{2}}
				e^{\left(-\kappa v^\circ(1-s) + \kappa C \right)^{\frac{2\beta}{2+\beta}}}\d s\,.
		\end{split}
\end{align}
Hence, via  inequalities ~\eqref{E:v-circ-bound-L-first}
and \eqref{E:v-circ-bound-L-second}, one deduces that
\begin{equation*}
	\int_{\rn} e^{\left(\kappa |u| \right)^{\frac{2\beta}{2+\beta}}}\dgn
	\le 2 \int_{0}^{\frac{1}{2}} \exp
		\left\{
			\left(
				\kappa\opnorm*{\frac{1}{I}}_{L^{\tilde B}(s,\frac12)}
				+ \kappa C
			\right)^{\frac{2\beta}{2+\beta}}
		\right\}\d s.
\end{equation*}
Inequality  \eqref{E:exponential-estimate} thus follows by the change of
variables $\Phi(t)=s$.

Next, assume that $X=m^B$. Assumption \eqref{E:gradient-X-1} and  the very
definition of Marcinkiewicz quasi-norm implies that
\begin{equation} \label{mar21}
	|\nabla u|^*(s) = |\nabla v|^*(s)
		\le B^{-1}\left( \frac{1}{s} \right)
		\quad\text{for $s\in(0,1)$}.
\end{equation}
By Proposition~\ref{P:med-mv},
\begin{equation*}
	|\med (u)-\mv (u)|
		\le \sqrt{\frac{\pi}{2}} \|\nabla u\|_{L^1\RG}
		\le \sqrt{\frac{\pi}{2}} \|B^{-1}(1/s)\|_{L^1(0,1)}
		= \sqrt{\frac{\pi}{2}} \int_{0}^{1}B^{-1}\left( \frac{1}{s} \right)\d s.
\end{equation*}
Thus, inequalities \eqref{E:uv-diff} and \eqref{E:exponential-estimate-u}
continue to hold, with $C = \sqrt{\frac{\pi}{2}} \int_{0}^{1}B^{-1}\left(
\frac{1}{s} \right)\d s$.  Hence, by Proposition~\ref{P:symubound} and
inequality \eqref{mar21},
\begin{align*}
	& \int_{\rn}
		e^{\left(\kappa |u| \right)^{\frac{2\beta}{2+\beta}}}
		\dgn
	\le 2 \int_{0}^{\frac{1}{2}} \exp
		\left\{
			\left(
				\frac{\kappa}{I(s)} \int_{0}^{s} B^{-1}\left( \frac{1}{r} \right)\,\d r
				+ \kappa \int_{s}^{\infty} \frac{B^{-1}\left( \frac{1}{r} \right)}{I(r)}\,\d r
				+ \kappa C
			\right)^{\frac{2\beta}{2+\beta}}
		\right\}
	\d s.
\end{align*}
Thereby, inequality \eqref{E:exponential-estimate}  follows via the change of
variables $r=\Phi(\tau)$ and $t=\Phi(s)$.
\end{proof}

An analogue of Lemma \eqref{L:Holder-estimates} under condition
\eqref{Linfinity} is provided by the last result of this section.

\begin{lemma}
\label{L:Holder-estimates-Linfty}
Let $\kappa >0$.  Then
\begin{equation} \label{E:exponential-estimate-infty}
	\int_{\rn} e^{(\kappa u)^2}\,\dgn
		\le \sqrt{\frac{2}{\pi}}
			\int_{0}^\infty e^{[\kappa \FF_{L^\infty}(t)]^2 - \frac{t^2}{2}}\,\d t
\end{equation}
for every weakly differentiable function $u$ obeying   \eqref{mu}  and such
that $\|\nabla u\|_\infty \le 1$. Here,
$\FF_{L^\infty}\colon(0,\infty)\to(0,\infty)$ denotes the function defined as
\begin{equation} \label{E:F-infty}
	\FF_{L^\infty}(t) =
		\begin{cases}
			t & \text{if $\med(u)=0$}
				\\
			t - 2\Phi'(t) - 2t\Phi(t)
				& \text{if $\mv(u)=0$}
		\end{cases}
		\quad\text{for $t>0$}.
\end{equation}
\end{lemma}

\begin{proof}
Assume first that $\med(u)=0$. We have that
\begin{align*}
	0 \leq u^\circ(s)
		& = \int_{s}^\frac12 -u^\circ{}'(r)I(r)\frac{\d r}{I(r)}
			\le \lVert-u^\circ{}'I\rVert_{L^\infty(0,1)}
			\int_{s}^\frac12 \frac{\d r}{I(r)}
			\\
		& \le \|\nabla u\|_{L^\infty\RG}
			\int_0^{\Phi^{-1}(s)} \frac{-\Phi'(t)}{I(\Phi(t))}\,\d t
		\le \Phi^{-1}(s) \quad \text{for $s\in(0,\tfrac12)$,}
\end{align*}
where we have made use of the change of variables $r=\Phi(t)$, of
Proposition~\ref{P:PSnorm} and of equation \eqref{june30}.  Similarly, thanks
to equation \eqref{symmetry},
\begin{align*}
	0\leq -u^\circ(s)
		& = \int_\frac12^s -u^\circ{}'(r)I(r)\frac{\d r}{I(r)}
		 \le \lVert-u^\circ{}'I\rVert_{L^\infty(0,1)}
			\int_\frac12^s \frac{\d r}{I(r)}
			\\
		& \le
			\int_{1-s}^\frac12 \frac{\d r}{I(r)}
				\le \Phi^{-1}(1-s)
			\quad\text{for $s\in(\tfrac12,1)$.}
\end{align*}
Therefore,
\begin{align*}
	\int_{\rn} e^{\left(\kappa |u| \right)^2} \dgn
	 \le \int_{0}^{\frac{1}{2}}
				e^{\left(\kappa u^\circ(s)\right)^2} \d s
			+ \int_{0}^{\frac{1}{2}}
				e^{\left(\kappa u^\circ(1-s) \right)^2} \d s
		 \le 2 \int_{0}^\frac12
				e^{\left(\kappa\Phi^{-1}(s) \right)^2} \d s
		= \sqrt{\frac{2}{\pi}} \int_{0}^\infty
				e^{\left(\kappa s\right)^2-\frac{s^2}{2}} \d s,
\end{align*}
namely \eqref{E:exponential-estimate-infty}.

Next, assume that $\mv(u)=0$. By \eqref{jun1},
\begin{equation*}
	u^\circ(s)
		= \int_{0}^1 \frac{\chi_{(s,1)}(r)-r}{I(r)}
				\bigl( -u^\circ{}'(r)I(r)\bigr)\,\d r
			\quad\text{for $s\in(0,1)$},
\end{equation*}
whence, by Proposition~\ref{P:PSnorm} and equation
\eqref{symmetry},
\begin{align} \label{jun2}
	\begin{split}
	|u^\circ(s)|
		 \le \lVert-u^\circ{}'I\rVert_{L^\infty(0,1)}
				\int_{0}^1 \frac{|\chi_{(s,1)}(r)-r|}{I(r)}\,\d r
		& \le \|\nabla u\|_{L^\infty\RG}
			\left( \int_{0}^s \frac{r}{I(r)}\,\d r
				+ \int_{s}^{1} \frac{1-r}{I(r)}\,\d r \right)
			\\
		& \le \int_{0}^s \frac{r}{I(r)}\,\d r
				+ \int_{0}^{1-s} \frac{r}{I(r)}\,\d r
			\quad\text{for $s\in(0,1)$}.
	\end{split}
\end{align}
By a change of variables, equation \eqref{june30}
and Fubini's theorem,
\begin{align*}
	\int_{0}^s \frac{r}{I(r)}\,\d r
		& = \int_{\Phi^{-1}(s)}^\infty \frac{\Phi(t)}{-I(\Phi(t))}\Phi'(t)\,\d t
			= \int_{\Phi^{-1}(s)}^\infty \Phi(t)\,\d t
			 = \int_{\Phi^{-1}(s)}^\infty \frac{1}{\sqrt{2\pi}}
					\int_{t}^\infty e^{-\frac{\tau^2}{2}}\,\d \tau\d t
			\\ & = \frac{1}{\sqrt{2\pi}} \int_{\Phi^{-1}(s)}^\infty
					e^{-\frac{t^2}{2}} \int_{\Phi^{-1}(s)}^t\,\d \tau\d t
			 = \frac{1}{\sqrt{2\pi}} \int_{\Phi^{-1}(s)}^\infty
					te^{-\frac{t^2}{2}}\,\d t
				- \Phi^{-1}(s)\frac{1}{\sqrt{2\pi}}
					\int_{\Phi^{-1}(s)}^\infty e^{-\frac{t^2}{2}}\,\d t
			\\ & = I(s) - \Phi^{-1}(s)s
			\quad\text{for $s\in(0,1)$}.
\end{align*}
Now, observe that $\Phi(-t)=1-\Phi(t)$ for $t\in \mathbb R$, whence
\begin{equation*}
	\Phi^{-1}(1-s)=-\Phi^{-1}(s)
		\quad\text{for $s\in(0,1)$}.
\end{equation*}
Thus, owing to equation  \eqref{symmetry}, inequality \eqref{jun2} yields
\begin{align}\label{june31}
	|u^\circ(s)|
		 \le I(s) - \Phi^{-1}(s)s + I(s) + \Phi^{-1}(s)(1-s)
			 = \Phi^{-1}(s) + 2I(s) - 2\Phi^{-1}(s)s
		 \quad\text{for $s\in(0,1)$.}
\end{align}
Altogether, by  the symmetry of the rightmost side of \eqref{june31} about $\frac 12$ and a change of variables,
\begin{align*}
	\int_{\rn} e^{\left(\kappa |u| \right)^2} \dgn
	= \int_{0}^{1} e^{\left(\kappa |u^\circ(s)|\right)^2} \d s
	&	\le  \int_{0}^1
			e^{\kappa^2\left(\Phi^{-1}(s) + 2I(s) - 2\Phi^{-1}(s)s \right)^2} \d s
		\\
	& = \sqrt{\frac{2}{\pi}} \int_{0}^\infty
		e^{\kappa^2\left(t+2I(\Phi(t))-t\Phi(t)\right)^2-\frac{t^2}{2}} \d t.
\end{align*}
Equation \eqref{E:exponential-estimate-infty} hence follows via \eqref{june30}.
\end{proof}

\section{Asymptotic expansions} \label{SE:back}

We are concerned here with various delicate asymptotic estimates for norm and
integral functionals, of exponential type, evaluated at the function $\Phi$
introduced  in \eqref{E:Phi-def}. Specifically, we deal with the functions
$\FF_{m^B}$ and $\FF_{L^B}$ defined by \eqref{E:F-mB} and \eqref{E:F-LB}.

Given a function  $\mathcal F$   defined in some neighborhood of
infinity, and $k\in \mathbb N$,   the notation
\begin{equation*}
	\mathcal F(t) = \mathcal E_1(t) + \cdots + \mathcal E_k(t) + \cdots
		\quad\text{as $t\to\infty$}
\end{equation*}
means that
\begin{equation*}
	\lim_{t\to\infty} \frac{\mathcal F(t)}
		{\mathcal E_1(t)} = 1
			\quad\text{if  $k=1$},\quad  \text{and} \quad  \lim_{t\to\infty} \frac{\mathcal F(t) - [\mathcal E_1(t) + \cdots + \mathcal E_{j}(t)]}
		{\mathcal E_{j+1}(t)} = 1
			\quad\text{for $1\le j\le k-1$, otherwise.}
\end{equation*}

Clearly, if
\begin{equation*}
	\mathcal F(t) = \mathcal E_1(t) + \mathcal E_2(t) + \cdots
		\quad\text{as $t\to\infty$},
\end{equation*}
and $\sigma >0$, then
\begin{equation} \label{E:power}
	[\mathcal F(t)]^\sigma
		= \mathcal E_1^\sigma(t) + \sigma \mathcal E_1^{\sigma-1}(t) \mathcal E_2(t)
			+ \cdots
		\quad\text{as $t\to\infty$.}
\end{equation}
Parallel notations will be used for asymptotic formulas as $t\to t_0$ and $t\to
t_0^+$, for some $t_0\in\R$.

We begin with two basic asymptotic  expansions contained in Lemmas~\ref{L:Filog}
and \ref{L:asyJ1} below. They easily follow from elementary considerations, via
applications of L'H\^opital's rule. Their proofs are omitted, for brevity.

\begin{lemma} \label{L:Filog}
Let $\Phi$ be given by \eqref{E:Phi-def}. Then
\begin{equation} \label{E:logPhi}
	-\log \Phi(t) = \frac{t^2}{2} +\log t + \cdots
		\quad\text{as $t\to\infty$}
\end{equation}
and
\begin{equation} \label{E:Phi-prime}
	-\Phi'(t) = t\Phi(t) + \frac{\Phi(t)}{t} + \cdots
		\quad\text{as $t\to\infty$}.
\end{equation}
\end{lemma}

\begin{lemma} \label{L:asyJ1}
Let $\beta>0$  and let $\Phi$ be given by
\eqref{E:Phi-def}. Assume that $B$ is any Young function satisfying condition \eqref{BN} for some  $N>0$.
Then
\begin{equation*}
	e^{\frac{t^2}{2}} \int_{t}^\infty
		B^{-1}\left( \frac{1}{\Phi(\tau)} \right)
			e^{-\frac{\tau^2}{2}}\,\d \tau
		= 2^{-\frac{1}{\beta}} t^{\frac{2}{\beta}-1} + \cdots
		\quad\text{as $t\to\infty$}.
\end{equation*}
\end{lemma}

The next result provides us with an expansion for the function $\FF_{m^B}$
defined by \eqref{E:F-mB}, which holds for every function $B$ fulfilling
assumption \eqref{BN} and for every $\beta>0$.

\begin{lemma} \label{L:B-integral}
Let $\beta>0$ and let $\Phi$ be given by \eqref{E:Phi-def}. Assume that $B$ is
a Young function satisfying condition \eqref{BN} for some $N>0$. Then
\begin{equation} \label{E:J2}
	\int_{0}^{t} B^{-1}\left( \frac{1}{\Phi(\tau)} \right)\d \tau
	= 2^{-\frac{1}{\beta}}\frac{\beta}{2+\beta} t^{\frac{2}{\beta}+1}
		+ \begin{cases}
			2^{-\frac{1}{\beta}} \frac{2}{2-\beta}t^{\frac{2}{\beta}-1}\log t + \cdots
				& \text{if $\beta\in(0,2)$}
				\\
			\frac{1}{2\sqrt{2}} (\log t)^2 + \cdots
				& \text{if $\beta = 2$}
				\\
			c  + \cdots,
				& \text{if $\beta\in(2,\infty)$}
		\end{cases}
		\quad\text{as $t\to\infty$.}
\end{equation}
Here, $c=c(B)\in \mathbb R$ is a   constant   depending on the global behavior
of $B$.  Consequently,
\begin{equation} \label{E:J2-power}
	\bigl[\kappab
		\FF_{m^B}(t)
	\bigr]^\frac{2\beta}{2+\beta}
		= \frac{t^2}{2}
		+ \begin{cases}
			\frac{2}{2-\beta} \log t + \cdots
				& \text{if $\beta\in(0,2)$}
				\\
			\frac12(\log t)^2 + \cdots
				& \text{if $\beta = 2$}
				\\
			c'\, t^{1-\frac{2}{\beta}} + \cdots
				& \text{if $\beta\in(2,\infty)$}
		\end{cases}
		\quad\text{as $t\to\infty$},
\end{equation}
where $\FF_{m^B}$ is defined by \eqref{E:F-mB}, $\kappab$ is given by
\eqref{E:kappab}, and $c'=c'(B)\in \mathbb R$ is a constant depending on the
global behavior of $B$.
\end{lemma}

\begin{proof}
Denote the integral on the left-hand side of \eqref{E:J2} by
$J(t)$ and define the function $g\colon (0, \infty) \to (0,\infty)$ as
\begin{equation*}
	g(t)=2^{-\ib}\frac{\beta}{2+\beta}t^{\iib+1}
		\quad\text{for $t>0$}.
\end{equation*}
Clearly $B^{-1}(r)=(\log r-\log N)^\ib$ near infinity.
By L'H\^opital's rule and Lemma~\ref{L:Filog},
\begin{equation*}
	\lim_{t\to\infty} \frac{J(t)}{g(t)}
		= \lim_{t\to\infty} \frac{B^{-1}\bigl(\frac{1}{\Phi(t)}\bigr)}{2^{-\ib}t^\iib}
		= \lim_{t\to\infty} \biggl( \frac{-\log	\Phi(t)-\log N}{\frac{t^2}{2}} \biggr)^\ib
		= 1.
\end{equation*}
In order to compute the second term in expansion \eqref{E:J2}, let us begin
by observing that, thanks to \eqref{E:logPhi} and \eqref{E:power},
\begin{align} \label{apr5}
	\begin{split}
	B^{-1}\left(\frac{1}{\Phi(t)}\right)
		 = ( - \log	\Phi(t) - \log N )^\ib
		&	= \left(\frac{t^2}{2}+\log t + \cdots\right)^\ib
	\\ &	 = 2^{-\ib} t^\iib + \iib 2^{-\ib} t^{\iib-2}\log t + \cdots
			\quad\text{as $t\to\infty$},
	\end{split}
\end{align}
for every $\beta>0$.  Let us now distinguish  the relevant three cases in
equation \eqref{E:J2}. Assume first that $\beta\in(0,2)$. Then
$t^{\iib-1}\to\infty$ as $t\to\infty$ and,  by L'H\^opital's rule and equation
\eqref{apr5},
\begin{equation*}
	\lim_{t\to\infty}
		\frac{J(t) - g(t)}
			{\frac{\beta}{2-\beta}t^{\frac{2}{\beta}-1}\log t}
	= \lim_{t\to\infty}
			\frac {B^{-1}\bigl(\frac{1}{\Phi(t)}\bigr)
				- 2^{-\ib} t^\iib}
				{t^{\iib-2}\log t + \frac{\beta}{2-\beta}t^{\iib-2}} = \iib 2^{-\ib}.
\end{equation*}
If $\beta=2$, then similarly, by \eqref{apr5},
\begin{equation*}
	\lim_{t\to\infty} \frac{J(t) -  \frac{1}{2\sqrt{2}} t^2 }
		{\frac12 (\log t)^2}
		 = \lim_{t\to\infty}
			\frac {B^{-1}\bigl(\frac{1}{\Phi(t)}\bigr)	- \frac{t}{\sqrt{2}}}
				{t^{-1} \log t}
		= \frac1{\sqrt{2}}.
\end{equation*}
Finally, when $\beta\in(2,\infty)$,
\begin{align*}
	\lim_{t\to\infty} \biggl(
		J(t) - 2^{-\ib} \frac{\beta}{2+\beta} t^{\iib+1} \biggr)
		& = \lim_{t\to\infty} \biggl(
			\int_{0}^{t} B^{-1}\biggl( \frac{1}{\Phi(\tau)} \biggr)\,\d \tau
			- 2^{-\ib} \int_{0}^{t} \tau^\iib\,\d \tau\biggr)
		    \\
		    & = \int_{0}^{\infty} \left(
			B^{-1}\left( \frac{1}{\Phi(\tau)} \right)
			-2^{-\ib} \tau^\iib\right)\d \tau,
\end{align*}
where the latter integral converges thanks to \eqref{apr5}.

Let us now focus on  \eqref{E:J2-power}. By equation \eqref{E:F-mB}, the
functional $\FF_{m^B}$ can be written as the sum of three terms. The first one
is, thanks to Lemma~\ref{L:asyJ1}, of a lower order than the  second one.  The
third term is just a constant. Therefore
\begin{equation*}
	\left[\kappab
		\FF_{m^B}(t)
	\right]^\frac{2\beta}{2+\beta}
	=
		\left(
		\kappab 2^{-\frac{1}{\beta}}\frac{\beta}{2+\beta} t^{\frac{2}{\beta}+1}
		+ \kappab\begin{cases}
			2^{-\frac{1}{\beta}} \frac{2}{2-\beta}t^{\frac{2}{\beta}-1}\log t
				&\\
			\frac{1}{2\sqrt{2}} (\log t)^2
				&\\
			c
				&
		\end{cases}
		+ \cdots
		\right)^\frac{2\beta}{2+\beta}
		\quad\text{as $t\to\infty$,}
\end{equation*}
for some constant $c$, where  the three cases in the brace correspond to $\beta
\in (0,2)$, $\beta =2$ and $\beta \in (2, \infty)$, respectively. The
conclusion   follows by \eqref{E:power}.
\end{proof}

The following lemma tells us that, if $\beta \in (2, \infty)$, then a function
$B$ as in Lemma \ref{L:B-integral} can be chosen in such a way that the
constant $c'$ appearing on the right-hand side of equation \eqref{E:J2-power}
attains prescribed negative values.

\begin{lemma} \label{L:B-integral-beta-large}
Assume that $\beta\in(2,\infty)$ and $\lambda>0$.  Then, given any  $N>0$,
there exists a Young function $B$ satisfying condition \eqref{BN}   and such
that
\begin{equation} \label{E:B-integral-lambda}
	\int_{0}^t B^{-1}\left( \frac{1}{\Phi(\tau)} \right)\d \tau
		\le 2^{-\frac{1}{\beta}}\frac{\beta}{2+\beta} t^{\frac{2}{\beta}+1}
			- \lambda + \cdots
		\quad\text{as $t\to\infty$.}
\end{equation}
Consequently, given any $\mu>0$, the function $B$ can be chosen in such a
way that
\begin{equation} \label{E:B-integral-lambda-power}
	\bigl[\kappab
		\FF_{m^{B}}(t)
	\bigr]^\frac{2\beta}{2+\beta}
		\le \frac{t^2}{2}
			- \mu t^{1-\frac{2}{\beta}} + \cdots
		\quad\text{as $t\to\infty$},
\end{equation}
where $\FF_{m^{B}}$ is defined by \eqref{E:F-mB} and
$\kappab$ is given by \eqref{E:kappab}.
\end{lemma}

\begin{proof}
Fix any $N>0$ and define the function $A\colon[0, \infty) \to [0, \infty)$ as
$A(t)=e^{t^\beta}$ for $t\ge 0$.
Given $t_0>0$, define the function $\AU\colon[0,\infty)\to[0,\infty)$
by
\begin{equation*}
	\AU(t) =
	\begin{cases}
		t\frac{A(t_0)}{t_0}
			& \text{for $t\in[0,t_0)$}
				\\
		A(t)
			& \text{for $t\in[t_0,\infty)$}.
	\end{cases}
\end{equation*}
Thus,
\begin{equation*}
	\AU^{-1}(\tau) =
	\begin{cases}
		\tau\frac{t_0}{A(t_0)}
			& \text{for $\tau\in[0,A(t_0))$}
				\\
		A^{-1}(t)
			& \text{for $\tau\in[A(t_0),\infty)$.}
	\end{cases}
\end{equation*}
Also set $B=N\AU$. Then $B$ is a Young function, provided that
$t_0$ is large enough, and
\begin{equation*}
	B^{-1}(\tau) = \AU^{-1}\left( \tau/N \right)
		\quad\text{for $\tau>0$}.
\end{equation*}
We may assume that $t_0$ is so large   that
$NA(t_0)>2$. Therefore,
\begin{equation*}
	\frac{1}{N\Phi(\tau)}\in [0,A(t_0))
		\quad\text{for $\tau\in [0,\tau(t_0))$,}
\end{equation*}
where we have set
\begin{equation} \label{apr7}
	\tau(t_0) = \Phi^{-1}\left( \frac{1}{NA(t_0)} \right).
\end{equation}
Now, set $\tau_0=0$ if $N \in (0,2]$ and
$\tau_0= \Phi^{-1}(1/N)$ if $N>2$. Then,
\begin{align*}
	\int_{0}^t B^{-1}\left( \frac{1}{\Phi(\tau)} \right)\d \tau
		& = \frac{t_0}{NA(t_0)} \int_{0}^{\tau(t_0)} \frac{\d \tau}{\Phi(\tau)}
			+ \int_{\tau(t_0)}^t \left( \log\frac{1}{N\Phi(\tau)} \right)^\ib\,\d \tau
			\\
		& = \frac{t_0}{NA(t_0)} \int_{0}^{\tau(t_0)} \frac{\d \tau}{\Phi(\tau)}
			+ \int_{\tau_0}^t \left( \log\frac{1}{N\Phi(\tau)} \right)^\ib\,\d \tau
			- \int_{\tau_0}^{\tau(t_0)} \left( \log\frac{1}{N\Phi(\tau)} \right)^\ib\,\d \tau
\end{align*}
for $t>\tau(t_0)$.  Next, by Lemma~\ref{L:B-integral},
\begin{equation} \label{apr11}
	\int_{\tau_0}^t \left( \log\frac{1}{N\Phi(\tau)} \right)^\ib\,\d \tau
		= 2^{-\ib} \frac{\beta}{2+\beta} t^{\iib+1} + c(t),
\end{equation}
where the function $c(t)\to c(\beta,N)$ as $t\to\infty$
and  $c(\beta,N)$ is a constant depending only on $\beta$ and $N$.
Thereby,
\begin{equation} \label{apr6}
	\int_{0}^t B^{-1}\left( \frac{1}{\Phi(\tau)} \right)\d \tau
		= 2^{-\ib} \frac{\beta}{2+\beta} t^{\iib+1}
			+ \lambda(t_0) + c(t),
\end{equation}
where
\begin{equation}\label{june33}
	\lambda(t_0)
		= \frac{t_0}{NA(t_0)} \int_{0}^{\tau(t_0)} \frac{\d \tau}{\Phi(\tau)}
			- \int_{\tau_0}^{\tau(t_0)} \left( \log\frac{1}{N\Phi(\tau)} \right)^\ib\,\d \tau.
\end{equation}
Let us now analyze the asymptotic behavior of $\lambda(t_0)$ as $t_0\to\infty$.
By L'H\^opital's rule,
\begin{equation*}
	\lim_{t\to\infty} \frac{\int_{0}^t \frac{\d \tau}{\Phi(\tau)}}{\frac{1}{t\Phi(t)}}
		= \lim_{t\to\infty} \frac{\frac{1}{\Phi(t)}}
				{\frac{-\Phi(t)-t\Phi'(t)}{t^2\Phi^2(t)}}
		= \lim_{t\to\infty} \frac{1}{-\frac{1}{t^2}-\frac{\Phi'(t)}{t\Phi(t)}}
		= 1,
\end{equation*}
where the last limit  holds thanks to equation \eqref{E:Phi-prime}.
Thus, by \eqref{apr7},
\begin{equation} \label{apr8}
	\frac{t_0}{NA(t_0)} \int_{0}^{\tau(t_0)} \frac{\d \tau}{\Phi(\tau)}
		= \frac{t_0}{\tau(t_0)} + \cdots
		\quad\text{as $t_0\to\infty$.}
\end{equation}
Also, by expansion \eqref{E:logPhi},
\begin{equation} \label{E:Phi-invers}
	\Phi^{-1}(s) = \sqrt{2\log\frac1s} + \cdots
		\quad\text{as $s\to 0^+$.}
\end{equation}
Therefore, by \eqref{apr7} and \eqref{E:Phi-invers},
\begin{equation} \label{apr9}
	\tau(t_0)
		= \Phi^{-1}\left( \frac{1}{N}e^{-t_0^\beta} \right)
		= \sqrt{2} t_0^{\frac{\beta}{2}} + \cdots
			\quad\text{as $t_0\to\infty$}.
\end{equation}
Coupling equations \eqref{apr8} and \eqref{apr9} tells us that
\begin{equation} \label{apr10}
	\frac{t_0}{NA(t_0)} \int_{0}^{\tau(t_0)} \frac{\d \tau}{\Phi(\tau)}
		= \frac{1}{\sqrt{2}} t_0^{1-\frac{\beta}{2}} + \cdots
			\quad\text{as $t_0\to\infty$}.
\end{equation}
Next, by equations \eqref{apr9} and \eqref{apr11},
\begin{align} \label{apr12}
	\begin{split}
	\int_{t_0}^{\tau(t_0)} \left( \log\frac{1}{N\Phi(\tau)} \right)^\ib\,\d \tau
		 = 2^{-\ib} \frac{\beta}{2+\beta} \tau(t_0)^{\iib+1} + \cdots
			 = \sqrt{2} \frac{\beta}{2+\beta} t_0^{\frac{\beta}{2}+1} + \cdots
			\quad\text{as $t_0\to\infty$}.
	\end{split}
\end{align}
Finally, on combining equations \eqref{june33}, \eqref{apr10} and \eqref{apr12} one deduces that
\begin{equation*}
	\lambda(t_0)
		= -\sqrt{2}\frac{\beta}{2+\beta} t_0^{\frac{\beta}{2}+1}
			+ \cdots
			\quad\text{as $t_0\to\infty$.}
\end{equation*}
This shows that   $\lambda(t_0)\to-\infty$ as $t_0\to\infty$.  Now, according
to \eqref{apr6}, given $\lambda>0$, we may choose $t_0$ so large   that
$\lambda(t_0)<-\lambda-c(\beta,N)$. Hence, equation \eqref{E:B-integral-lambda}
follows.

Let us now prove equation \eqref{E:B-integral-lambda-power}.  The function
$\FF_{m^{B}}$ agrees with the sum of the integral \eqref{E:B-integral-lambda}
with two more terms.  The first additional term obeys
\begin{equation*}
	e^{\frac{t^2}{2}} \int_{t}^\infty
		B^{-1}\left( \frac{1}{\Phi(\tau)} \right)
			e^{-\frac{\tau^2}{2}}\,\d \tau
		\to 0
		\quad\text{as $t\to\infty$},
\end{equation*}
by Lemma~\ref{L:asyJ1}, and the second term satisfies
\begin{align*}
	\int_{0}^{1}B^{-1}\left( \frac{1}{s} \right)\d s
		& = \int_{0}^{\frac{1}{NA(t_0)}} A^{-1}\left( \frac{1}{sN} \right)\,\d s
				+ \frac{t_0}{NA(t_0)}\int_{\frac{1}{NA(t_0)}}^1 \frac{\d s}{s}
			\\
		& = \int_{0}^{\frac{1}{NA(t_0)}} \log\left( \frac{1}{sN} \right)^\ib\,\d s
				+ \frac{1}{N}t_0\, e^{-t_0^\beta} \log\left(Ne^{t_0^\beta}\right).
\end{align*}
Note that both addends on the rightmost side of the last equation approach $0$
as $t_0\to\infty$.  Equation  \eqref{E:B-integral-lambda-power} then follows
via \eqref{E:B-integral-lambda} and \eqref{E:power}.
\end{proof}

In the remaining part of this section, we focus on an asymptotic estimate for
the function $\FF_{L^B}$ given by \eqref{E:F-LB}. This is the content of
Lemma~\ref{L:I-norm-sharp}. Its proof in its turn requires some preliminary
asymptotic expansions that are the objective of a few lemmas.

Lemmas~\ref{L:Psi}--\ref{L:a-invers} below are stated without proofs. They can
be derived via simple arguments relying upon L'H\^opital's rule.

\begin{lemma}\label{L:Psi}
Let $\sigma\in[-\tfrac12,\infty)$ and $d>1$.  Define the function
$\Psi_\sigma\colon (d,\infty)\to[0,\infty)$ as
\begin{equation}
\label{E:Psi}
	\Psi_\sigma(t) = \int_{d}^{t} (\tau^2-1)^\sigma\,\d \tau
		\quad\text{for $t>d$.}
\end{equation}
If $\sigma\in(-\frac12,\infty)$, then
\begin{equation}
	\Psi_\sigma(t)
		= \frac{1}{2\sigma+1} t^{2\sigma+1}
			-
			\begin{cases}	
			c + \cdots
				& \text{if $\sigma\in(-\tfrac12,\tfrac12)$}
				\\
			\frac{1}{2}\log t + \cdots
				& \text{if $\sigma=\tfrac12$}
				\\
			\frac{\sigma}{2\sigma-1} t^{2\sigma-1} + \cdots
				& \text{if $\sigma\in(\tfrac12,\infty)$}
			\end{cases}
			\quad\text{as $t\to\infty$,}
\end{equation}
where $c\in\R$ is a constant depending on $\sigma$ and $d$.

If $\sigma=-\frac12$, then
\begin{equation*}
	\Psi_\sigma(t)=\log t + \cdots
		\quad\text{as $t\to\infty$.}
\end{equation*}
\end{lemma}

\begin{lemma} \label{L:Upsilon}
Let $\sigma\in[-\tfrac12,\infty)$ and let  $d>1$. Define the
function $\Upsilon_\sigma\colon (d,\infty)\to[0,\infty)$ as
\begin{equation}
\label{E:Upsilon}
	\Upsilon_\sigma(t) = \int_{d}^t (\tau^2-1)^\sigma \log(\tau^2-1)\,\d \tau
		\quad\text{for $t>d$.}
\end{equation}
Then
\begin{equation}
	\Upsilon_\sigma (t)=
		\begin{cases}
			\frac{2}{2\sigma+1} t^{2\sigma+1} \log t
				+ \cdots
				& \text{if $\sigma\in(-\frac12,\infty)$}
				\\
			(\log t)^2
			+ \cdots
				& \text{if $\sigma = -\frac12$}
		\end{cases}
	\quad\text{as $t\to\infty$.}
\end{equation}
\end{lemma}

\begin{lemma} \label{L:a-invers}
Let $B$ be a Young function obeying \eqref{BN} for some $N>0$, and let
$b\colon[0,\infty)\to[0,\infty)$ be the left-continuous function such that
$B(t)=\int_{0}^t b(\tau)\,\d\tau$ for $t>0$.  Then
\begin{equation*} 
	b^{-1}(t)
		= (\log t)^\frac{1}{\beta}
		+ \frac{1-\beta}{\beta^2}
			(\log t)^{\frac{1}{\beta}-1}\log\log t
		+ \cdots
			\quad\text{as $t\to\infty$}.
\end{equation*}
\end{lemma}

\begin{lemma} \label{L:Orlicz-norm-inf}
Let $A$ be a Young function and let $(\RR, \nu)$ be a probability space. Then
\begin{equation} \label{E:Orlicz-norm-inf}
	\opnorm{\phi}_{L^{A}}
		\le \inf_{k>0}
			\left\{
				\frac{1}{k}+\frac{1}{k}\int_{\RR} A(k|\phi |)\,\d\nu
			\right\}.
\end{equation}
\end{lemma}

\begin{proof}
By Young's inequality \eqref{E:Young-ineq},
\begin{equation*}
	\int_{\RR} |\phi\psi|\,\d\nu
		\le \frac{1}{k} \int_{\RR} \tilde A(|\psi|)\,\d\nu
			+ \frac{1}{k} \int_{\RR} A(k|\phi|)\,\d\nu \quad \text{for $k>0$.}
\end{equation*}
Therefore
\begin{equation*}
	\opnorm{\phi}_{L^{A}}
		= \sup\biggl\{ \int_{\RR} |\phi\psi|\,\d\nu:
				\int_{\RR} \tilde A(|\psi|)\,\d\nu \le 1\biggr\}
		\le \frac{1}{k} + \frac{1}{k} \int_{\RR} A(k|\phi|)\,\d\nu \quad \text{for $k>0$,}
\end{equation*}
whence \eqref{E:Orlicz-norm-inf} follows.
\end{proof}

\begin{lemma} \label{L:I-norm-sharp}
Let $B$ be a Young function satisfying condition \eqref{BN} for some constant  $N>0$.
Then
\begin{equation} \label{E:I-norm-bound}
		\begin{split}
		\normIB
			\le  2^{-\ib} \frac{\beta}{2+\beta} t^{\iib+1}
				- 2^{-\ib}
				\begin{cases}
					\frac{2}{2-\beta}t^{\iib-1}\log t + \cdots
						& \text{if $\beta\in(0,2)$}
						\\
					\frac{1}{2}(\log t)^2 + \cdots
						& \text{if $\beta=2$}
						\\
					c + \cdots
						& \text{if $\beta\in(2,\infty)$}
				\end{cases}
		\quad\text{as $t\to\infty$,}
	\end{split}
\end{equation}
for a suitable  constant $c=c(B)\in\R$.
Consequently,
\begin{equation} \label{E:I-norm-power}
	\bigl[ \kappab \FF_{L^B}(t) \bigr]^\frac{2\beta}{2+\beta}
		\le \frac{t^2}{2} -
		\begin{cases}
			\frac{2}{2-\beta}\log t + \cdots
				& \text{if $\beta\in(0,2)$}
				\\
			\frac{1}{2}(\log t)^2 + \cdots
				& \text{if $\beta=2$}
				\\
			c'\, t^{1-\frac{2}{\beta}} + \cdots
				& \text{if $\beta\in(2,\infty)$}
		\end{cases}
		\quad\text{as $t\to\infty$},
\end{equation}
where $\FF_{L^B}$ is defined by \eqref{E:F-LB} and $c'=c'(B)\in\R$ is a suitable constant.
\end{lemma}

\begin{proof}
By Lemma~\ref{L:Orlicz-norm-inf},
\begin{equation*}
	\normIB
		\le \inf_{k>0}
		\left\{
			\frac{1}{k}+\frac{1}{k}
				\int_{\Phi(t)}^{\frac12} \tilde{B}\left( \frac{k}{I(s)} \right)\d s
		\right\},
\end{equation*}
whence, by the change of variables $s=\Phi(\tau)$,
\begin{equation} \label{feb1}
	\normIB
		\le \inf_{k>0}
			\left\{
				\frac{1}{k}
				+ \frac{1}{k\sqrt{2\pi}} \int_{0}^t
					\tilde B\left( k\sqrt{2\pi} e^{\frac{\tau^2}{2}} \right)
					e^{-\frac{\tau^2}{2}}\,\d \tau
			\right\}.
\end{equation}
Choose $k=\frac{1}{\sqrt{2\pi}}e^{-\frac{\sigma(t)^2}{2}}$ in the expression in
braces on the right-hand side of inequality \eqref{feb1}, where  $\sigma\colon(e,
\infty) \to (0, \infty)$  is the function defined as
\begin{equation} \label{E:tau}
	\sigma(t) =
	\begin{cases}
		\sqrt{2\Bigl(\frac{2}{\beta}-1\Bigr)\log t},
			& \text{if $\beta\in(0,2)$}
			\\
		\sqrt{2\log\log t}
			& \text{if $\beta=2$}
			\\
		1
			& \text{if $\beta\in(2,\infty)$}
	\end{cases}
	\quad\text{for $t>e$}.
\end{equation}
Notice that
\begin{equation}\label{july10}
	\lim_{t \to \infty} \frac{t}{\sigma(t)}
		= \infty.
\end{equation}
This choice of $k$ in \eqref{feb1} yields
\begin{equation} \label{sep2}
	\normIB
		\le \sqrt{2\pi} e^{\frac{\sigma(t)^2}{2}}
			+ e^{\frac{\sigma(t)^2}{2}} \int_{0}^t
				\tilde B\Bigl( e^{\frac{\tau^2}{2}-\frac{\sigma(t)^2}{2}} \Bigr)
				e^{-\frac{\tau^2}{2}}\,\d \tau
		\quad\text{for $t>e$}.
\end{equation}
Owing to Lemma~\ref{L:a-invers} and to the fact that $\tilde B(t)\le
tb^{-1}(t)$ for $t>0$, one has that
\begin{equation*}
	\tilde B(t)
		\le t(\log t)^\ib
			+ \frac{1-\beta}{\beta^2}
				t(\log t)^{\ib-1}\log\log t
			+ \cdots
		\quad\text{as $t\to\infty$.}
\end{equation*}
Consequently, given $\varepsilon>0$,
there exists a constant $C>0$ such that
\begin{equation}
\label{E:A-conjugate-bound-C}
	\tilde B(t)
		\le t(\log t)^\ib
			+ K_\varepsilon \frac{1-\beta}{\beta^2}
				t(\log t)^{\ib-1}\log\log t
			+ C
		\quad\text{for $t>e$,}
\end{equation}
where  $K_\varepsilon=1+\varepsilon$ if $\beta\in(0,1]$ and
$K_\varepsilon=1-\varepsilon$ if $\beta\in(1,\infty)$.  On enlarging, if
necessary, the value of $C$, we may also assume   that $\tilde B(t)\le C$ for
$t\in(0,e)$.  From inequalities  \eqref{sep2} and
\eqref{E:A-conjugate-bound-C}, we infer that
\begin{equation} \label{E:I-norm-upper-bound}
	\normIB
		 \le \sqrt{2\pi}(1+C) e^{\frac{\sigma(t)^2}{2}}
			+ J_1(t)
			+ K_\varepsilon \frac{1-\beta}{\beta^2} J_2(t),
\end{equation}
where we have set
\begin{equation*} 
	J_1(t)
		= \int_{\sigma(t)}^t
			\left( \frac{\tau^2}{2}-\frac{\sigma(t)^2}{2} \right)^\ib \!\d \tau
\end{equation*}
and
\begin{equation*} 
	J_2(t)
		= \int_{2\sigma(t)}^t
			\left( \frac{\tau^2}{2}-\frac{\sigma(t)^2}{2} \right)^{\ib-1}
				\!\log\left( \frac{\tau^2}{2}-\frac{\sigma(t)^2}{2} \right) \d \tau
\end{equation*}
for $t>e$.  Let us estimate the term $J_1$.  By a change of variables, we
obtain that
\begin{align} \label{E:I1-homo}
	\begin{split}
	J_1(t)
		& = \int_{1}^{\frac{t}{\sigma(t)}}
			\left( \frac{\sigma(t)^2}{2} \tau^2 - \frac{\sigma(t)^2}{2} \right)^\ib \sigma(t)\,\d \tau
			= 2^{-\ib} \sigma(t)^{\iib+1}
			\int_{1}^{\frac{t}{\sigma(t)}} (\tau^2-1)^\frac1\beta \,\d \tau
			\\
		& = 2^{-\ib} \sigma(t)^{\iib+1}
			\Psi_{\frac1\beta}\Bigl( \frac{t}{\sigma(t)} \Bigr)
	\quad\text{for $t>e$},
	\end{split}
\end{align}
where $\Psi_{\frac1\beta}$ is defined as in  \eqref{E:Psi}.
From Lemma~\ref{L:Psi} and equation \eqref{july10}  one can infer that
\begin{equation} \label{E:Psi-ib}
	\begin{split}
	\Psi_{\frac1\beta}\Bigl( \frac{t}{\sigma(t)} \Bigr)
		= \frac{\beta}{2+\beta} \Bigl(\frac{t}{\sigma(t)} \Bigr)^{\iib+1}
			-
			\begin{cases}
			\dfrac{1}{2-\beta} \Bigl(\dfrac{t}{\sigma(t)} \Bigr)^{\iib-1} + \cdots
				& \text{if $\beta\in(0,2)$}
				\\
			\frac{1}{2}\log\dfrac{t}{\sigma(t)} + \cdots
				& \text{if $\beta=2$}
				\\
			c_1 + \cdots
				& \text{if $\beta\in(2,\infty)$}
			\end{cases}
	\quad\text{as $t\to\infty$.}
	\end{split}
\end{equation}
Coupling  \eqref{E:I1-homo} with \eqref{E:Psi-ib} yields
\begin{equation} \label{E:I1-bound}
	J_1(t)
		= 2^{-\ib}\frac{\beta}{2+\beta} t^{\iib+1}
			-2^{-\ib}
			\begin{cases}
				\frac{1}{2-\beta} t^{\iib-1} \sigma(t)^2 + \cdots
					& \text{if $\beta\in(0,2)$}
					\\
				\frac{1}{2}\sigma(t)^2 \log\dfrac{t}{\sigma(t)} + \cdots
					& \text{if $\beta=2$}
					\\
				c_1 \sigma(t)^{\iib+1} + \cdots
					& \text{if $\beta\in(2,\infty)$}
			\end{cases}
	\quad\text{as $t\to\infty$.}
\end{equation}
Equation \eqref{E:tau} and   estimates \eqref{E:I1-bound}  tell us that
\begin{equation} \label{E:I1-bound-final}
	J_1(t)
		= 2^{-\ib}\frac{\beta}{2+\beta} t^{\iib+1}
			-2^{-\ib}
			\begin{cases}
				\iib t^{\iib-1} \log t + \cdots
					& \text{if $\beta\in(0,2)$}
					\\
				\frac{1}{2}\log t\log\log t + \cdots
					& \text{if $\beta=2$}
					\\
				c_1 + \cdots
					& \text{if $\beta\in(2,\infty)$}
			\end{cases}
	\quad\text{as $t\to\infty$}.
\end{equation}
Let us next consider $J_2$. If $\beta\in(0,2]$, then, by a change of variables,
\begin{align*}
	J_2(t)
		& = \sigma(t) \left( \frac{\sigma(t)^2}{2} \right)^{\frac{1}{\beta}-1}
			\int_{2}^\frac{t}{\sigma(t)} (\tau^2-1)^{\frac{1}{\beta}-1}
				\left( \log\frac{\sigma(t)^2}{2} + \log(\tau^2-1) \right)\d \tau
			\\
		& = 2^{1-\ib}\sigma(t)^{\frac{2}{\beta}-1}
				\left(
					\log\frac{\sigma(t)^2}{2}\,\Psi_{\ib-1}\Bigl( \frac{t}{\sigma(t)}\Bigr)
					+ \Upsilon_{\ib-1}\Bigl( \frac{t}{\sigma(t)}\Bigr)
				\right)
	\quad\text{for $t>e$}.
\end{align*}
Here, $\Upsilon_{\ib-1}$ is defined according to  \eqref{E:Upsilon}.  Thanks to
Lemmas~\ref{L:Psi} and~\ref{L:Upsilon} and equation \eqref{july10} one has
that
\begin{equation*}
	\Psi_{\ib-1}\Bigl( \frac{t}{\sigma(t)}\Bigr) =
		\begin{cases}
			\frac{\beta}{2-\beta}\Bigl( \dfrac{t}{\sigma(t)}\Bigr)^{\iib-1}
			+ \cdots &\text{if $\beta\in(0,2)$}
				\\
			\log \dfrac{t}{\sigma(t)}+\cdots
			& \text{if $\beta=2$}
		\end{cases}
	\quad\text{as $t\to\infty$}
\end{equation*}
and
\begin{equation*}
	\Upsilon_{\ib-1}\Bigl( \frac{t}{\sigma(t)}\Bigr) =
			\begin{cases}
				\frac{2\beta}{2-\beta}\Bigl( \dfrac{t}{\sigma(t)}\Bigr)^{\iib-1}
				\log\dfrac{t}{\sigma(t)}+ \cdots
				& \text{if $\beta\in(0,2)$}
					\\
				\Bigl(\log \dfrac{t}{\sigma(t)}\Bigr)^2 + \cdots
				& \text{if $\beta=2$}
 			\end{cases}
	\quad\text{as $t\to\infty$.}
\end{equation*}
Consequently, on making use of  \eqref{E:tau}, one deduces that
\begin{equation} \label{E:I2-bound-final}
	J_2(t) =
	\begin{cases}
		2^{-\ib}\frac{4\beta}{2-\beta}t^{\iib-1}\log t + \cdots
			& \text{if $\beta\in(0,2)$}
				\\
		2^{\frac{1}{2}}(\log t)^2 + \cdots
			& \text{if $\beta=2$}
	\end{cases}
	\quad\text{as $t\to\infty$.}
\end{equation}
If $\beta\in(2,\infty)$, one can verify that  $J_2(t)$ has a finite limit
 as $t\to\infty$.
Furthermore,
\begin{equation} \label{E:exp-tau}
	e^{\frac{\sigma(t)^2}{2}}
		=
		\begin{cases}
			t^{\iib-1}
				& \text{if $\beta\in(0,2)$}
				\\
			\log t
				& \text{if $\beta=2$}
				\\
			e^{\frac{1}{2}}
				& \text{if $\beta\in(2,\infty)$}
		\end{cases}
	\quad\text{for $t>e$}.
\end{equation}
Therefore, equations \eqref{E:I-norm-upper-bound}, \eqref{E:I1-bound-final},
\eqref{E:I2-bound-final} and \eqref{E:exp-tau} enable us to conclude that, if $\beta\in(0,1]$, then
\begin{align*}
	\normIB
		& \le 2^{-\ib} \frac{\beta}{2+\beta} t^{\iib+1}
			-	2^{-\ib} \iib t^{\iib-1} \log t
			+ (1+\varepsilon)\frac{1-\beta}{\beta^2}
				2^{-\ib} \frac{4\beta}{2-\beta}t^{\iib-1}
			\log t
			+ \cdots
		\\
		& = 2^{-\ib} \frac{\beta}{2+\beta} t^{\iib+1}
			-	2^{-\ib} \left(\frac{2}{2-\beta} - \frac{4\varepsilon(1-\beta)}{\beta(2-\beta)}\right) t^{\iib-1} \log t
			+ \cdots
	\quad\text{as $t\to\infty$,}
\end{align*}
 and that, if $\beta\in(1,2)$, then
\begin{equation*}
	\normIB
		 \le 2^{-\ib} \frac{\beta}{2+\beta} t^{\iib+1}
			-	2^{-\ib} \left(\frac{2}{2-\beta} + \frac{4\varepsilon(1-\beta)}{\beta(2-\beta)}\right) t^{\iib-1} \log t
		+ \cdots
	\quad\text{as $t\to\infty$.}
\end{equation*}
Thus, thanks to the arbitrariness of $\varepsilon$,
equation \eqref{E:I-norm-bound} follows in the case when $\beta \in (0,2)$.
If $\beta=2$, then, by \eqref{E:I-norm-upper-bound}, \eqref{E:I1-bound-final},
\eqref{E:I2-bound-final} and \eqref{E:exp-tau},
\begin{equation*}
	\normIB
		\le 2^{-\frac32}t^2
			- \frac{1+\varepsilon}{4}2^{\frac{1}{2}}(\log t)^2
			+ \cdots
	\quad\text{as $t\to\infty$.}
\end{equation*}
Hence, the arbitrariness of $\varepsilon$ enables us to deduce
\eqref{E:I-norm-bound} for $\beta =2$.  Finally, If $\beta\in(2,\infty)$, then
equation  \eqref{E:I-norm-bound} holds by \eqref{E:I-norm-upper-bound},
\eqref{E:I1-bound-final}, and by the boundedness of $J_2$.

By definition \eqref{E:F-LB}, estimate \eqref{E:I-norm-power}  follows from
equations \eqref{E:I-norm-bound} and \eqref{E:power}.
\end{proof}

\section{Proofs of the main results} \label{SE:proofs}

The proofs of our main results  exploit relations between the assumption in
integral form \eqref{E:integral-form-M} and that in norm form
\eqref{E:nabla-Orlicz}, and  the relations in \eqref{AtoM} between strong and
weak norms. Some steps of such proofs are stated as separate intermediate
results, in order to avoid unnecessary repetitions.

The next three lemmas provide us with links between conditions
\eqref{E:integral-form-M} and \eqref{E:nabla-Orlicz}.

\begin{lemma} \label{L:modular-implies-norm-to-B-is-M}
Let $\beta >0$.  Assume that $B$ is a Young function satisfying condition
\eqref{BN} for some $N\in(0,1)$.  Then there exists a constant $M>1$ such that
\begin{equation} \label{E:phi-norm-1}
	\|\phi\|_{L^B\RG} \le 1
\end{equation}
for every function $\phi\in\MM(\rn)$ fulfilling
\begin{equation} \label{E:phi-modular-M}
	\int_{\rn} e^{|\phi|^\beta}\dgn \le M.
\end{equation}
\end{lemma}

\begin{proof}
Denote by $E\colon[0, \infty) \to [0, \infty)$ the convex envelope of the function
$e^{t^\beta}-1$. Namely,  $E$
is the largest convex function not exceeding $e^{t^\beta}-1$ on $[0,\infty)$.
Given $M >1$,   define the function  $B_M\colon[0,\infty)\to[0,\infty)$ as
\begin{equation*}
	B_M(t) = \frac{E(t)}{M-1}
		\quad\text{for $t\in[0,\infty)$.}
\end{equation*}
Observe, that $B_M$ is a Young function.
Now, if $\phi$ fulfills \eqref{E:phi-modular-M}, then
\begin{equation*}
	\int_{\rn} B_M(|\phi|)\,\dgn
		\le \frac{1}{M-1} \int_{\rn} \left( e^{|\phi|^\beta} -1 \right)\dgn
		\leq 1,
\end{equation*}
whence, by the definition of Luxemburg norm,
$\|\phi\|_{L^{B_M}\RG}\le 1$.
Owing to H\"older's inequality \eqref{E:Holder} and to \eqref{E:Orl_char},
\begin{equation} \label{jun12}
	\int_{\rn} |\phi|\,\dgn
		 \le \|\phi\|_{L^{B_M}\RG}
			\opnorm{1}_{L^{\tilde{B_M}}\RG}
			\le B_M^{-1}(1).
\end{equation}
Notice that
\begin{equation} \label{jun11}
	\lim _{M\to 1^+} B_M^{-1}(1) = 0,
\end{equation}
inasmuch as $E(t)\to 0$ as $t\to 0^+$.  Now, assume that $B$ obeys \eqref{BN}
for $t\in(t_0,\infty)$.  Since $B$ is convex and vanishes at $0$,
\begin{equation*}
	B(t) \le N
	\begin{cases}
		t\frac{e^{t_0^\beta}}{t_0}
			& \text{for $t\in[0,t_0)$}
			\\
		e^{t^\beta}
			& \text{for $t\in[t_0,\infty)$}.
	\end{cases}
\end{equation*}
Therefore, by \eqref{E:phi-modular-M} and \eqref{jun12},
\begin{align*}
	\int_{\rn} B(|\phi|)\,\dgn
		& \le N\frac{e^{t_0^\beta}}{t_0}
			\int_{\{|\phi|\le t_0\}} |\phi|\,\dgn
			+ N \int_{\{|\phi|> t_0\}} e^{|\phi|^\beta}\,\dgn
			\\
		& \le N\frac{e^{t_0^\beta}}{t_0}
			\int_{\rn} |\phi|\,\dgn
			+ N \int_{\rn} e^{|\phi|^\beta}\,\dgn
			 = N\left( \frac{e^{t_0^\beta}}{t_0}B_M^{-1}(1) + M \right).
\end{align*}
Observe, that, owing to \eqref{jun11}, the expression in  brackets on the
rightmost side  converges to $1$ as $M$ tends to $1^+$. Consequently, since
$N\in(0,1)$, $M$ can be chosen so close to $1$  that
\begin{equation*}
	N\left( \frac{e^{t_0^\beta}}{t_0}(\log M)^\frac{1}{\beta} + M \right)
		\le 1,
\end{equation*}
whence \eqref{E:phi-norm-1} follows.
\end{proof}

\begin{lemma} \label{L:norm-implies-modular-to-B-is-M}
Let $\beta >0$.  Assume that $B$ is a Young function satisfying condition
\eqref{BN} for some $N>0$.  Then there exists a constant $M> 1$ such that
inequality \eqref{E:phi-modular-M} holds for every function $\phi\in\MM(\rn)$
fulfilling condition \eqref{E:phi-norm-1}.
\end{lemma}

\begin{proof}
Let $t_0\in(0,\infty)$ be such that $B$ fulfills condition \eqref{BN} for
$t\in(t_0,\infty)$.  By the definition of Luxemburg norm, assumption
\eqref{E:phi-norm-1} is equivalent to
\begin{equation*}
	\int_{\rn} B(|\phi|)\,\dgn \le 1.
\end{equation*}
Thereby
\begin{align*}
	\int_{\rn} e^{|\phi|^\beta}\,\dgn
		\le \int_{\rn} e^{t_0^\beta}\,\dgn
			+ \int_{\{|\phi|\ge t_0\}} e^{|\phi|^\beta}\,\dgn
		 \le e^{t_0^\beta} + \frac 1N \int_{\rn} B(|\phi|)\,\dgn
			= e^{t_0^\beta} + \frac1N.
\end{align*}
The conclusion follows by choosing $M=e^{t_0^\beta}+1/N$.
\end{proof}

\begin{lemma} \label{L:modular-implies-norm-to-M-is-B}
Let $\beta>0$ and $M>1$. Then there exists a Young function $B$ satisfying
condition \eqref{BN} for some $N>0$, and such that inequality
\eqref{E:phi-norm-1} holds for every function $\phi\in\MM(\rn)$ obeying
condition \eqref{E:phi-modular-M}.
\end{lemma}

\begin{proof}
Given $t_0>0$,  define the function $A\colon[0,\infty)\to[0,\infty)$ by
\begin{equation*}
	A(t) =
	\begin{cases}
		t\frac{e^{t_0^\beta}}{t_0}
			& \text{for $t\in [0,t_0)$}
		\\
		e^{t^\beta}
			& \text{for $t\in [t_0,\infty)$}.
	\end{cases}
\end{equation*}
Clearly, $t_0$ can be chosen   large  enough for  $A$ to be
convex. Set  $N=1/(M+e^{t_0^\beta})$ and let $B=NA$. We claim that $B$
is a Young function with the required properties.  Indeed, if $\phi$ is any
 function obeying \eqref{E:phi-modular-M}, then
\begin{align*}
	\int_{\rn} B(|\phi|)\,\dgn
		\le N \int_{\{|\phi|\ge t_0\}} e^{|\phi|^\beta}\,\dgn
				+ N \int_{\{|\phi|<t_0\}} e^{t_0^\beta}\,\dgn
				\le N\left(M + e^{t_0^\beta}\right) = 1.
\end{align*}
Hence, inequality  \eqref{E:phi-norm-1} follows by the very
definition of Luxemburg norm.
\end{proof}

Propositions~\ref{P:mB-critical-beta-large} and \ref{P:LB-critical-beta-large}
below are of use in the proof of part (2.ii) of Theorems~\ref{T:norm-form} and
\ref{T:weak-form}.

\begin{proposition} \label{P:mB-critical-beta-large}
Let $\beta\in(2,\infty)$.  Then, given any  $N>0$, there exist a Young function
$B$ satisfying condition \eqref{BN} for some $t_0$, and a constant
$C=C(\beta,N,t_0)$ such that
\begin{equation}\label{july11}
	\int_{\rn} e^{\left( \kappab |u|\right)^\frac{2\beta}{2+\beta}}\,\dgn
		\le C
\end{equation}
for every function $u\in\WexpLb$ fulfilling \eqref{mu}
and \eqref{E:nabla-Marcinkiewicz-m}.
\end{proposition}

\begin{proof}
Fix any $\mu>0$ and let $B$ be a Young function for which equation
\eqref{E:B-integral-lambda-power} holds.  By Lemma~\ref{L:Holder-estimates},
\begin{equation} \label{apr15}
	\int_{\rn}
		e^{\left(\kappab |u| \right)^{\frac{2\beta}{2+\beta}}}\,\dgn
		\le \sqrt{\frac{2}{\pi}}
			\int_{0}^{\infty} e^{\left[ \kappab\FF_{m^{B}}(t) \right]^{\frac{2\beta}{2+\beta}}
				-\frac{t^2}{2}}\,\d t,
\end{equation}
where $\FF_{m^{B}}$ is given by \eqref{E:F-mB}. Owing to
Lemma~\ref{L:B-integral-beta-large},
\begin{equation*}
	\bigl[\kappab
		\FF_{m^{B}}(t)
	\bigr]^\frac{2\beta}{2+\beta}
		- \frac{t^2}{2}
		\le  - \mu t^{1-\frac{2}{\beta}} + \cdots
		\quad\text{as $t\to\infty$}.
\end{equation*}
This ensures that  the integral on the right-hand side of \eqref{apr15}
converges. Inequality \eqref{july11} thus follows from \eqref{apr15}.
\end{proof}

\begin{proposition} \label{P:LB-critical-beta-large}
Let $\beta\in(2,\infty)$. Then given any $N>0$, there exist a Young function
$B$ satisfying condition \eqref{BN} and a sequence of functions $\{u_k\}
\subset \WexpLb$, such that $\med (u_k)=\mv (u_k)=0$ and
\begin{equation} \label{E:nabla-u-t}
	\|\nabla u_k\|_{L^{B}\RG} \le 1
\end{equation}
for $k \in \mathbb N$,
satisfying
\begin{equation}\label{july12}
	\lim_{k\to\infty}
		\int_{\rn} e^{\left( \kappab |u_k|\right)^\frac{2\beta}{2+\beta}}\,\dgn
			= \infty.
\end{equation}
\end{proposition}

\begin{proof}
\let\sigma\tau
Set $A(t)=e^{t^\beta}$ for $t\ge 0$ and let $t_0>0$.
Define the function  $\AL\colon[0,\infty)\to[0,\infty)$ by
\begin{equation*}
	\AL(t) =
	\begin{cases}
		0
			& \text{for $t\in[0,t_0')$}
			\\
		a(t_0)(t-t_0) + A(t_0)
			& \text{for $t\in[t_0',t_0)$}
			\\
		A(t)
			& \text{for $t\in[t_0,\infty)$},
	\end{cases}
\end{equation*}
where $a$ denotes the derivative of $A$ and
\begin{equation*}
	t_0' = t_0 - \frac{A(t_0)}{a(t_0)}.
\end{equation*}
Notice that $t_0' \in(0,t_0)$. Given $N>0$, set  $B=N\AL$. Clearly $B(t) =
Ne^{t^\beta}$ for $t>t_0$. On denoting by $\aL$ the left-continuous function
such that $\AL(t)=\int_0^t\aL (\tau)\,\d\tau$ for $t\geq0$, one has that
\begin{equation*}
	\aL(t) =
	\begin{cases}
		0
			& \text{for $t\in[0,t_0')$}
			\\
		a(t_0)
			& \text{for $t\in[t_0',t_0)$}
			\\
		a(t)
			& \text{for $t\in[t_0,\infty)$}
	\end{cases}
	\quad\text{and}\quad
	\aL^{-1}(\tau) =
	\begin{cases}
		0
			& \text{for $\tau=0$}
			\\
		t_0'
			& \text{for $\tau\in(0,a(t_0)]$}
			\\
		a^{-1}(\tau)
			& \text{for $\tau\in(a(t_0),\infty)$}.
	\end{cases}
\end{equation*}
For each $k \in \mathbb N$,  define the function $u_k\colon\rn\to\R$ by
\begin{equation*}
	u_k(x) = \sgn x_1
	\begin{cases}
		\displaystyle
		\int_{0}^{|x_1|} \aL^{-1}\Bigl( e^{\frac{\tau^2}{2}} \Bigr)\,\d \tau
			& \text{for $|x_1|<k$}
			\\
		\displaystyle
		\int_{0}^{k} \aL^{-1}\Bigl( e^{\frac{\tau^2}{2}} \Bigr)\,\d \tau
			& \text{for $|x_1|\ge k$}.
	\end{cases}
\end{equation*}
Clearly $\med(u_k)=\mv(u_k)=0$ for $k \in \mathbb N$.
Next, set
\begin{equation*}
	\sigma(t_0)=\sqrt{2\log a(t_0)}.
\end{equation*}
If  $k>\sigma(t_0)$, then
\begin{align*}
	|\nabla u_k(x)|
	 =
	\begin{cases}
		\aL^{-1}\Bigl( e^{\frac{|x_1|^2}{2}} \Bigr)
			& \text{for $|x_1|<k$}
			\\
		0
			& \text{for $|x_1|>k$}
	\end{cases}
	=
	\begin{cases}
		t_0'
			& \text{for $|x_1|<\sigma(t_0)$}
			\\
		a^{-1}\Bigl( e^{\frac{|x_1|^2}{2}} \Bigr)
			& \text{for $\sigma(t_0)<|x_1|<k$}
			\\
		0
			& \text{for $k<|x_1|$.}
	\end{cases}
\end{align*}
Therefore,
\begin{align*}
	\int_{\rn} B(|\nabla u_k|)\,\dgn
		& = \frac{2N}{\sqrt{2\pi}} \int_{0}^{k}
			\AL\left( \aL^{-1}\Bigl( e^{\frac{\tau^2}{2}} \Bigr) \right)
				e^{-\frac{\tau^2}{2}}\,\d \tau
				\\
		& = N\sqrt{\frac{2}{\pi}}
			\int_{0}^{\sigma(t_0)} \AL(t_0')
				e^{-\frac{\tau^2}{2}}\,\d \tau
			+ N\sqrt{\frac{2}{\pi}}\int_{\sigma(t_0)}^{k}
			A\left( a^{-1}\Bigl( e^{\frac{\tau^2}{2}} \Bigr) \right)
				e^{-\frac{\tau^2}{2}}\,\d \tau.
\end{align*}
By   the definition of $\AL$,  one has that $\AL(t_0')=0$.  Thus the first
integral on the rightmost side of the last equation vanishes.  On the other
hand, we have that
\begin{equation*}
	A(t)
		= Ne^{t^\beta}
		= \ib t^{1-\beta}a(t)
		\quad\text{for $t>0$},
\end{equation*}
whence, owing to Lemma~\ref{L:a-invers},
\begin{equation*}
	A(a^{-1}(t))
		= \ib t\left( a^{-1}(t) \right)^{1-\beta}
		= \ib t(\log t)^{\frac{1}{\beta}-1}
			+ \cdots
			\quad\text{as $t\to\infty$}
\end{equation*}
and the second integral converges as $k\to\infty$ since $\beta\in(2,\infty)$.
Thus, if $k>\sigma(t_0)$, then
\begin{equation*}
	\int_{\rn} B(|\nabla u_k|)\,\dgn
		\le M(t_0),
\end{equation*}
where we have set
\begin{equation*}
	M(t_0) = N\sqrt{\frac{2}{\pi}}\int_{\sigma(t_0)}^{\infty}
			A\left( a^{-1}\Bigl( e^{\frac{\tau^2}{2}} \Bigr) \right)
				e^{-\frac{\tau^2}{2}}\,\d \tau.
\end{equation*}
Since $M(t_0)$ tends to $0$ as  $t_0\to\infty$,
we may choose  $t_0$ sufficiently large that
$\int_{\rn}B(|\nabla u_k|)\,\dgn\le 1$. Hence, inequality
\eqref{E:nabla-u-t} holds for $k>\sigma(t_0)$.

Let us now focus on equation \eqref{july12}. We have that
\begin{equation} \label{E:mar1}
	\int_{\rn} e^{(\kappab|u_k|)^\frac{2\beta}{2+\beta}}\,\dgn
		\ge \int_{\{|x_1|>k\}}
			e^{(\kappab|u_k|)^\frac{2\beta}{2+\beta}}\,\dgn
		= 2\Phi(k)\exp\left\{
				\left( \kappab
					\int_{0}^{k} \aL^{-1}\Bigl(e^{\frac{\tau^2}{2}}\Bigr) \,\d \tau
				\right)^\frac{2\beta}{2+\beta}
			\right\}.
\end{equation}
Next,
\begin{equation*}
	\int_{0}^{k} \aL^{-1}\Bigl(e^{\frac{\tau^2}{2}}\Bigr)\,\d \tau
		= \left( t_0-\frac{A(t_0)}{a(t_0)} \right)
			\int_{0}^{\sigma(t_0)}\d \tau
		+ \int_{\sigma(t_0)}^{k}
			a^{-1}\Bigl(e^{\frac{\tau^2}{2}}\Bigr)\,\d \tau.
\end{equation*}
By Lemma~\ref{L:a-invers}, there exists $\tau_0>0$ such that
\begin{equation*}
	a^{-1}(\tau)
		\ge (\log \tau)^\frac{1}{\beta}
		- (\log \tau)^{\frac{1}{\beta}-1}\log\log \tau
		\quad\text{for $\tau>\tau_0$}.
\end{equation*}
Assume, in addition, that $t_0$ obeys $\sigma(t_0)>\tau_0$. Then
\begin{align*}
	\int_{\sigma(t_0)}^{k}
		a^{-1}\bigl(e^{\frac{\tau^2}{2}}\bigr)\,\d \tau
		& \ge \int_{\sigma(t_0)}^{k}
				\left( \frac{\tau^2}{2} \right)^\frac{1}{\beta}\d \tau
			- \int_{\sigma(t_0)}^{k}
				\left( \frac{\tau^2}{2} \right)^{\frac{1}{\beta}-1}\log\frac{\tau^2}{2}\,\d \tau
			\\
		& \ge 2^{-\frac{1}{\beta}}\frac{\beta}{2+\beta} k^{\frac{2}{\beta}+1}
			- \sqrt{2}\frac{\beta}{2+\beta}
				\bigl(\log a(t_0)\bigr)^{\frac{1}{\beta}+\frac{1}{2}}
			- \int_{\sigma(t_0)}^{\infty}
				\left( \frac{\tau^2}{2} \right)^{\frac{1}{\beta}-1}\log\frac{\tau^2}{2}\,\d \tau,
\end{align*}
since the last integral converges.
Consequently, if $k>\sigma(t_0)$, then
\begin{equation*}
	\int_{0}^{k} \aL^{-1}\bigl(e^{\frac{\tau^2}{2}}\bigr)\,\d \tau
		\ge 2^{-\frac{1}{\beta}}\frac{\beta}{2+\beta} k^{\frac{2}{\beta}+1}
			+ \lambda(t_0),
\end{equation*}
where we have set
\begin{equation}\label{july14}
	\lambda(t_0)
		= \sqrt{2}\left( t_0-\frac{A(t_0)}{a(t_0)} \right)
				\bigl(\log a(t_0)\bigr)^\frac{1}{2}
			- \sqrt{2}\frac{\beta}{2+\beta}
				\bigl(\log a(t_0)\bigr)^{\frac{1}{\beta}+\frac{1}{2}}
			- \int_{\sigma(t_0)}^{\infty}
				\left( \frac{\tau^2}{2} \right)^{\frac{1}{\beta}-1}\log\frac{\tau^2}{2}\,\d \tau.
\end{equation}
Let us analyze the behavior of $\lambda(t_0)$ as $t_0\to\infty$. One has that
\begin{equation*}
	t_0\bigl(\log a(t_0)\bigr)^\frac{1}{2}
		= t_0 \bigl(t_0^\beta+(\beta-1)\log t_0+\log\beta\bigr)^\frac{1}{2}
		= t_0^{\frac{\beta}{2}+1} + \cdots
		\quad\text{as $t_0\to\infty$}
\end{equation*}
and
\begin{equation*}
	\bigl(\log a(t_0)\bigr)^{\frac{1}{\beta}+\frac{1}{2}}
		= \bigr(t_0^\beta+(\beta-1)\log t_0+\log\beta\bigr)^{\frac{1}{\beta}+\frac{1}{2}}
		= t_0^{\frac{\beta}{2}+1} + \cdots
		\quad\text{as $t_0\to\infty$}.
\end{equation*}
The remaining terms on the right-hand side of \eqref{july14} are of a lower
order, since both $A(t_0)/a(t_0)$  and the integral approach $0$  as
$t_0\to\infty$. Thus,
\begin{align*}
	\lambda(t_0)
		& = \sqrt{2}\frac{2}{2+\beta} t_0^{\frac{\beta}{2}+1}
			+ \cdots
			\quad\text{as $t_0\to\infty$.}
\end{align*}
As a consequence, given any $\lambda>0$,  we may choose $t_0$ so large that
$\lambda(t_0)>\lambda$. This choice ensures that
\begin{align*}
	\left(
		\kappab \int_{0}^{k} \aL^{-1}\bigl(e^{\frac{\tau^2}{2}}\bigr)\,\d \tau
	\right)^\frac{2\beta}{2+\beta}
		\ge \left( \kappab
			2^{-\frac{1}{\beta}}\frac{\beta}{2+\beta} k^{\frac{2}{\beta}+1}
			+ \kappab\lambda \right)^\frac{2\beta}{2+\beta}
		= \frac{k^2}{2} + 2^{\frac{1}{\beta}}\lambda k^{1-\frac{2}{\beta}}
				+ \cdots
				\quad\text{as $k\to\infty$}.
\end{align*}
Therefore, by inequality \eqref{E:mar1}
and relation \eqref{E:logPhi},
\begin{align}\label{july15}
\begin{split}
	\int_{\rn} e^{(\kappab|u_k|)^\frac{2\beta}{2+\beta}}\,\dgn
		& \ge \exp\left\{ \log\Phi(k) +
				\left( \kappab
					\int_{0}^{k} \aL^{-1}\Bigl(e^{\frac{\tau^2}{2}}\Bigr) \,\d \tau
				\right)^\frac{2\beta}{2+\beta}
			\right\}
			\\
		& \ge \exp\left\{ \left( -\frac{k^2}{2}-\log k+\cdots \right)
				+ \left( \frac{k^2}{2} + 2^{\frac{1}{\beta}}\lambda k^{1-\frac{2}{\beta}}
					+ \cdots \right)\right\}
				\\
		& = \exp\left\{ 2^{\frac{1}{\beta}}\lambda
			k^{1-\frac{2}{\beta}} + \cdots \right\}.
\end{split}
\end{align}
Equation \eqref{july12} follows, since the rightmost side of equation
\eqref{july15}   tends to infinity as $k\to\infty$.
\end{proof}

The next proposition implies
part (1.ii) of Theorem~\ref{T:weak-form}.

\begin{proposition} \label{P:MB-critical-beta-small}
Let $\beta\in(0,2]$.  Assume that $B$ is a Young function satisfying condition
\eqref{BN} for some $N>0$.  Then there exists a function $u\in\WexpLb$,
fulfilling condition \eqref{E:nabla-Marcinkiewicz-M},  such that
$\med(u)=\mv(u)=0$ and
\begin{equation} \label{apr1}
	\int_{\rn} e^{\left( \kappab |u|\right)^\frac{2\beta}{2+\beta}}
		\,\dgn = \infty.
\end{equation}
\end{proposition}

\begin{proof}
Let $\tau_0$ be such that $B(\tau)=Ne^{\tau^\beta}$ for $\tau\in(\tau_0,\infty)$.
Set
\begin{equation*}
	s_0=\min\left\{ \frac{1}{2}, \frac{1}{N}e^{-\tau_0^\beta} \right\}
\end{equation*}
and $t_0=\Phi^{-1}(s_0)$.
Define the function $u\colon\rn\to\R$ by
\begin{equation} \label{E:u-def}
	u(x) = \sgn x_1
	\begin{cases}
		\displaystyle\int_{\Phi(|x_1|)}^{\Phi(t_0)} \frac{g(s)}{I(s)}\,\d s
			&\text{for $|x_1| \ge t_0$}
			\\
		0
			&\text{for $|x_1| < t_0$},
	\end{cases}
\end{equation}
where $g\colon (0,s_0) \to [0, \infty)$ is the function given by
\begin{equation*}
	g(s) = \left( \log \frac{1}{Ns} \right)^\frac{1}{\beta}
		- \frac{1}{\beta} \left( \log\frac{1}{Ns} \right)^{\frac{1}{\beta}-1}
		\quad\text{for $s\in(0,s_0)$}.
\end{equation*}
Observe that $g$ is decreasing,  provided that $N$ is  chosen sufficiently
large. Clearly $\med (u)=\mv (u)=0$.  Furthermore,
\begin{equation*}
	|\nabla u (x)| =
	\begin{cases}
		g(\Phi (|x_1|))
			&\text{for $|x_1| \ge t_0$}
			\\
		0
			&\text{for $|x_1| <t_0$.}
	\end{cases}
\end{equation*}
Hence,
\begin{equation*}
	|\nabla u|^*(s) =
	\begin{cases}
		g(s)
			&\text{for $s\in(0,s_0)$}
			\\
		0
			&\text{for $s\in(s_0,1)$},
	\end{cases}
\end{equation*}
and
\begin{equation*}
	|\nabla u|^{**}(s) =
	\begin{cases}
		\left( \log\frac{1}{Ns} \right)^\frac{1}{\beta}
			&\text{for $s\in(0,s_0)$}
			\\
		\frac{s_0}{s}\left( \log\frac{1}{Ns_0} \right)^\frac{1}{\beta}
			&\text{for $s\in(s_0,1)$.}
	\end{cases}
\end{equation*}
Therefore
\begin{align*}
	\|\nabla u\|_{M^B\RG}
		& = \sup_{s\in(0,1)}
			\frac{|\nabla u|^{**}(s)}{B^{-1}\left( \frac{1}{s} \right)}
			\\
		& =\max\left\{
			\sup_{s\in(0,s_0)}
				\frac{\left( \log\frac{1}{Ns} \right)^\frac{1}{\beta}}
				{B^{-1}\left( \frac{1}{s} \right)},
			s_0\left( \log\frac{1}{Ns_0} \right)^\frac{1}{\beta}
			\sup_{s\in(s_0,1)} \frac{1}{s B^{-1}\left( \frac{1}{s} \right)}
			\right\}
		= 1,
\end{align*}
where the last equality holds since the function $\frac{1}{s
B^{-1}\left(1/s\right)}$ is non-increasing, inasmuch as $B$ is a Young
function.

It remains to prove \eqref{apr1}. We have that
\begin{align} \label{apr2}
	\begin{split}
	\int_{\rn} e^{\left( \kappab |u| \right)^\frac{2\beta}{2+\beta}}\,\dgn
		& \ge \sqrt{\frac{2}{\pi}} \int_{t_0}^{\infty}  e^{\left(
				\kappab \int_{\Phi(t)}^{\Phi(t_0)} \frac{g(s)}{I(s)}\,\d s
			\right)^\frac{2\beta}{2+\beta}
			- \frac{t^2}{2}}\,\d t
			\\
		& = \sqrt{\frac{2}{\pi}} \int_{t_0}^{\infty}  e^{\left(
				\kappab \int_{t_0}^{t} g(\Phi(\tau))\,\d \tau
			\right)^\frac{2\beta}{2+\beta}
			- \frac{t^2}{2}}\,\d t,
	\end{split}
\end{align}
where we have made use of the change of variables $s=\Phi(\tau)$.
On the other hand, by the definition of~$g$,
\begin{equation} \label{apr3}
	\int_{t_0}^{t} g(\Phi(\tau))\,\d \tau
		=
		\int_{t_0}^{t} \left( \log \frac{1}{N\Phi(\tau)} \right)^\frac{1}{\beta}\,\d \tau
		- \frac{1}{\beta}
		\int_{t_0}^{t} \left(\log\frac{1}{N\Phi(\tau)}\right)^{\frac{1}{\beta}-1}\,\d \tau
		\quad\text{for $t\in(t_0,\infty)$}.
\end{equation}
Lemma~\ref{L:B-integral} implies that
\begin{equation*}
	\int_{t_0}^{t} \left( \log \frac{1}{N\Phi(\tau)} \right)^\frac{1}{\beta}\,\d \tau
		= 2^{-\frac{1}{\beta}}\frac{\beta}{2+\beta} t^{\frac{2}{\beta}+1}
			+
			\begin{cases}
				2^{-\frac{1}{\beta}} \frac{2}{2-\beta} t^{\frac{2}{\beta}-1}\log t
					+ \cdots
					& \text{if $\beta\in(0,2)$}
					\\
				\frac{1}{2\sqrt{2}}(\log t)^2
					+ \cdots
					& \text{if $\beta=2$}
			\end{cases}
\end{equation*}
as $t\to\infty$, and, by analogous computations,
\begin{equation*}
	\int_{t_0}^{t} \left( \log \frac{1}{N\Phi(\tau)} \right)^{\frac{1}{\beta}-1}\,\d \tau
		=
		\begin{cases}
			2^{1-\frac{1}{\beta}} \frac{\beta}{2-\beta} t^{\frac{2}{\beta}-1}
				+ \cdots
				& \text{if $\beta\in(0,2)$}
				\\
			\sqrt{2} \log t
				+ \cdots
				& \text{if $\beta=2$}
		\end{cases}
	\quad\text{as $t\to\infty$}.
\end{equation*}
Thereby,
the second integral on the right-hand side of  \eqref{apr3} is of a lower order than the first one
as $t\to\infty$. From  \eqref{E:Phi-prime} we thus infer that
\begin{equation*}
	\left( \kappab \int_{t_0}^{t} g(\Phi(\tau))\,\d \tau \right)^\frac{2\beta}{2+\beta}
		- \frac{t^2}{2}
		=
		\begin{cases}
			\frac{2}{2-\beta} \log t
				+ \cdots
				& \text{if $\beta\in(0,2)$}
				\\
			\frac{1}{2} (\log t)^2
				+ \cdots
				& \text{if $\beta=2$}
		\end{cases}
	\quad\text{as $t\to\infty$.}
\end{equation*}
Equation  \eqref{apr1} hence follows via
 \eqref{apr2}.
\end{proof}

Our last preparatory result is contained in the following proposition.  It is a
key step in the proofs of parts (1.ii) and (2.iii) of
Theorems~\ref{T:integral-form} and \ref{T:norm-form}, and of part (2.iii) of
Theorem~\ref{T:weak-form}, .

\begin{proposition} \label{P:LB-supercritical-beta-all}
Let $\beta>0$, $M>1$ and $\kappa>\kappa_\beta$.  Assume that $B$ is a Young
function satisfying condition
\eqref{BN} for some $N>0$.  Then there exists a function
$u\in\WexpLb$ such that $\med(u)=\mv(u)=0$,
\begin{equation}\label{july18}
	\|\nabla u\|_{L^B\RG} \le 1\,,
	\quad \quad
	\int_{\rn} e^{|\nabla u|^\beta}\,\dgn \le M\,
\end{equation}
and
\begin{equation}\label{july19}
	\int_{\rn} e^{\left( \kappa |u|\right)^\frac{2\beta}{2+\beta}}\,\dgn
		= \infty.
\end{equation}
\end{proposition}

\begin{proof}
Let   $u$ be the function defined as in \eqref{E:u-def},
where
\begin{equation*}
	g(s) = 	\biggl(\lambda \log\frac{\Phi(t_0)}{s}\biggr)^\frac1\beta
		\quad\text{for $s\in(0,\Phi(t_0))$,}
\end{equation*}
for some  $t_0>0$ and $\lambda\in(0,1)$ to  be specified
later. Clearly $\med (u)=\mv (u) =0 $.  Also,
\begin{equation*}
	|\nabla u(x)| =
	\begin{cases}
		0
			& \text{for $|x_1|<t_0$}
			\\
		\Bigl(\lambda \log \frac{\Phi(t_0)}{\Phi(|x_1|)} \Bigr)^\frac1\beta
			& \text{for $|x_1|\ge t_0$}.
	\end{cases}
\end{equation*}
Let $\tau_0>0$ be such that $B(\tau)=Ne^{\tau^\beta}$ for $\tau\ge \tau_0$.
Then,
\begin{equation}\label{july17}
	\int_{\rn} B(|\nabla u|)\,\dgn
		\le \int_{\{0<|\nabla u|\le \tau_0\}} B(\tau_0) \,\dgn
			+ N \int_{\{|x_1| > t_0\}} e^{|\nabla u|^\beta}\,\dgn.
\end{equation}
Since the support of $\nabla u$ agrees with  the union of the  two half-spaces $\{x_1>t_0\}$ and $\{x_1<-t_0\}$,
\begin{equation*}
	\int_{\{0<|\nabla u|\le \tau_0\}} B(\tau_0)\,\dgn
		\le 2B(\tau_0)\Phi(t_0).
\end{equation*}
Furthermore,
\begin{align}
	\begin{split} \label{feb3}
	\int_{\{|x_1| > t_0\}} e^{|\nabla u|^\beta}\,\dgn
		& = 2\int_{t_0}^\infty \int_{\R^{n-1}}
				e^{\lambda \log \frac{\Phi(t_0)}{\Phi(x_1)}}\,\dgn(x)
			= \frac{2}{\sqrt{2\pi}} \int_{t_0}^{\infty}
				e^{\lambda \log \frac{\Phi(t_0)}{\Phi(x_1)}} e^{-\frac{x_1^2}{2}}\, \d x_1
			\\
		& = 2 \int_{0}^{\Phi(t_0)}
				\biggr( \frac{\Phi(t_0)}{s}\biggr)^\lambda \,\d s
			= 2 \Phi(t_0) \int_{0}^{1} \frac{\d s}{s^\lambda}
			= \frac{2}{1-\lambda} \Phi(t_0).
		\end{split}
\end{align}
Combining inequalities \eqref{july17}--\eqref{feb3} yields
\begin{equation*}
	\int_{\rn} B(|\nabla u|)\,\dgn
		\le 2B(\tau_0)\Phi(t_0) + \frac{2N}{1-\lambda} \Phi(t_0).
\end{equation*}
Therefore, if the constants $t_0$ and $\lambda$ are chosen in such a way that
\begin{equation} \label{E:a-and-lambda-condition-norm}
	2\Phi(t_0)\left(B(\tau_0) + \frac{N}{1-\lambda} \right) \le 1,
\end{equation}
then $\int_{\rn} B(|\nabla u|)\,\dgn \le 1$ and, by the definition of Luxemburg
norm, we have that $\|\nabla u\|_{L^B\RG}\le 1$, namely the first inequality
in \eqref{july18} holds.  As far as the second one is concerned, we infer from
\eqref{feb3} that
\begin{equation*}
	\int_{\rn} e^{|\nabla u|^\beta}\,\dgn
		\le \int_{\rn} e^0 \,\dgn
			+ \int_{\{|x_1| > t_0\}} e^{|\nabla u|^\beta}\,\dgn
		\le 1 + \frac{2}{1-\lambda} \Phi(t_0).
\end{equation*}
Hence, if $\lambda$ and $t_0$ also obey
\begin{equation} \label{E:a-and-lambda-condition-integral}
	\frac{2}{1-\lambda} \Phi(t_0)\le M-1,
\end{equation}
then $\int_{\rn} e^{|\nabla u|^\beta}\,\dgn \le M$. Thus, the second inequality in
\eqref{july18} is fulfilled as well.

In order to prove property \eqref{july19}, observe that,
similarly to \eqref{apr2},
\begin{equation} \label{mar1}
	\int_{\rn} e^{(\kappa|u|)^\frac{2\beta}{2+\beta}}\,\dgn
		\ge \sqrt{\frac{2}{\pi}} \int_{t_0}^{\infty}
			e^{
			\left(\kappa\lambda^\frac{1}{\beta}
				\int_{t_0}^{t}
				\left(\log\frac{\Phi(t_0)}{\Phi(\tau)}\right)^\frac1\beta \,\d \tau
			\right)^{\frac{2\beta}{2+\beta}} - \frac{t^2}{2}
			}\,\d t.
\end{equation}
Now, let $A$ be  any Young function
such that $A(t)=e^{t^\beta}$ near infinity.  By L'H\^opital's rule,
\begin{equation*}
	\lim_{t\to\infty}
		\frac{\displaystyle\int_{t_0}^{t}
			\left(\log\frac{\Phi(t_0)}{\Phi(\tau)}\right)^\frac1\beta \,\d \tau}
		{ \displaystyle \int_{0}^{t} A^{-1}\left( \frac{1}{\Phi(\tau)} \right)\d \tau }
		= \lim_{t\to\infty} \frac{\left(\log\frac{\Phi(t_0)}{\Phi(t)}\right)^\frac1\beta}
		{\left(\log\frac{1}{\Phi(t)}\right)^{\frac1\beta}}
		= 1.
\end{equation*}
Thereby, thanks to Lemma~\ref{L:B-integral},
\begin{equation*}
	\Biggl( \kappab
		\int_{t_0}^{t}
		\biggl(\log\frac{\Phi(t_0)}{\Phi(\tau)}\biggr)^\frac1\beta \,\d \tau
	\Biggr)^{\frac{2\beta}{2+\beta}}
		= \Biggl(\kappab\int_{0}^{t} A^{-1}\left( \frac{1}{\Phi(\tau)} \right)\d \tau
			+\cdots\Biggr)^\frac{2\beta}{2+\beta}
		= \frac{t^2}{2} + \cdots
			\quad\text{as $t\to\infty$}.
\end{equation*}
As a consequence,
\begin{equation} \label{dec7}
	\left(\kappa\lambda^\frac{1}{\beta}
				\int_{t_0}^{t}
				\left(\log\frac{\Phi(t_0)}{\Phi(\tau)}\right)^\frac1\beta \,\d \tau
			\right)^{\frac{2\beta}{2+\beta}} - \frac{t^2}{2}
	= \left( \lambda^\frac{2}{2+\beta}
			\left( \frac{\kappa}{\kappab} \right)^\frac{2\beta}{2+\beta}-1
		\right) \frac{t^2}{2} + \cdots
			\quad\text{as $t\to\infty$}.
\end{equation}
If $\kappa>\kappa_\beta$, one can choose $\lambda\in(0,1)$ sufficiently close
to 1 in such a way that the constant multiplying $\frac{t^2}2$ on the right-hand side of
\eqref{dec7} is positive. With this choice of $\lambda$, property
\eqref{july19} follows from \eqref{mar1}. Then, we choose $t_0$  large enough
for \eqref{E:a-and-lambda-condition-norm} and
\eqref{E:a-and-lambda-condition-integral} to hold, whence the inequalities in
\eqref{july18} hold as well.
\end{proof}

We are now in a position to accomplish the proofs of our main results.

\begin{proof}[Proof of Theorem~\ref{T:norm-form}]
\hypertarget{LB-critical-all-B-finite}
Assume that $\kappa>0$.
By Lemma~\ref{L:Holder-estimates},
\begin{equation} \label{apr16}
	\int_{\rn}
		e^{\left(\kappa |u| \right)^{\frac{2\beta}{2+\beta}}}\,\dgn
		\le \sqrt{\frac{2}{\pi}}
			\int_{0}^{\infty} e^{\left[ \kappa\FF_{L^B}(t) \right]^{\frac{2\beta}{2+\beta}}
				-\frac{t^2}{2}}\,\d t
\end{equation}
for any weakly differentiable $u$ obeying \eqref{E:nabla-Orlicz} and
\eqref{mu}, where $\FF_{L^B}$ is given by \eqref{E:F-LB}.  On the other hand,
Lemma~\ref{L:I-norm-sharp} tells us that
\begin{equation*}
	\bigl[\kappab
		\FF_{L^B}(t)
	\bigr]^\frac{2\beta}{2+\beta}
		- \frac{t^2}{2}
		= -
		\begin{cases}
		\frac{2}{2-\beta}\log t + \cdots
				& \text{if $\beta\in(0,2)$}
				\\
			\frac{1}{2}(\log t)^2 + \cdots
				& \text{if $\beta=2$}
		\end{cases}
		\quad\text{as $t\to\infty$}.
\end{equation*}
Hence, if $\beta\in(0,2]$ and $\kappa=\kappab$, then
the integral on the right-hand side of \eqref{apr16} converges. This proves part (1.i).
If $\beta\in(2,\infty)$ one infers, from Lemma~\ref{L:I-norm-sharp}, that
\begin{equation*}
	\bigl[\kappa
		\FF_{L^B}(t)
	\bigr]^\frac{2\beta}{2+\beta}
		= \left( \frac{\kappa}{\kappab} \right)^\frac{2\beta}{2+\beta}
			\frac{t^2}{2}
			+ \cdots
		\quad\text{as $t\to\infty$.}
\end{equation*}
The right-hand side of \eqref{apr16}
converges for $\kappa<\kappab$ also in this case, thus proving assertion (2.i).

Assertions (1.ii) and (2.iii) follow
from Proposition~\ref{P:LB-supercritical-beta-all}.

It remains to prove  statement (2.ii), corresponding to  the critical case
$\kappa=\kappab$ when $\beta\in(2,\infty)$.  The negative part is proved in
Proposition~\ref{P:LB-critical-beta-large}, whereas the positive part follows from
Proposition~\ref{P:mB-critical-beta-large} and from inequalities
\eqref{AtoM}.
\end{proof}

\begin{proof}[Proof of Theorem~\ref{T:integral-form}]
By Lemma~\ref{L:modular-implies-norm-to-M-is-B}, there exists a Young function
$B$ satisfying property \eqref{BN} for some $N>0$ and such that every weakly
differentiable function satisfying \eqref{E:integral-form-M} obeys
\eqref{E:nabla-Orlicz}.  Claims (1.i) and (2.i) then follow via
Theorem~\ref{T:norm-form} parts (1.i) and (2.i), respectively.

Assertions (1.ii) and (2.iii) are covered by
Proposition~\ref{P:LB-supercritical-beta-all}.

Let us now deal with the critical case when $\kappa=\kappab$.
Fix $N\in(0,1)$. By Theorem~\ref{T:norm-form}, part (2.ii),
there exists a Young function $B$
satisfying property \eqref{BN} such that
\begin{equation} \label{apr18}
	\int_{\rn} e^{(\kappab |u|)^{\frac{2\beta}{2+\beta}}} \,\dgn
		\le C
\end{equation}
for some $C=C(\beta,B)$ and for every weakly differentiable $u$ obeying \eqref{mu}
and \eqref{E:nabla-Orlicz}.  By Lemma~\ref{L:modular-implies-norm-to-B-is-M},
there exists a constant $M>1$ such that  condition \eqref{E:integral-form-M}
implies \eqref{E:nabla-Orlicz}. This proves the positive part of (2.ii).

Analogously, Theorem~\ref{T:norm-form}, part (2.ii), ensures that for any $N>0$
there exists a Young function satisfying condition \eqref{BN} for which
inequality \eqref{apr18} fails whatever $C$ is, as $u$ ranges over all weakly
differentiable functions obeying \eqref{E:nabla-Orlicz} and \eqref{mu}.  By
Lemma~\ref{L:norm-implies-modular-to-B-is-M}, there exists a constant $M>1$
such that inequality \eqref{E:nabla-Orlicz} is satisfied for any $u$ obeying
\eqref{E:integral-form-M}.  The proof is now complete.
\end{proof}

\begin{proof}[Proof of Theorem~\ref{T:weak-form}]
\hypertarget{MB-subcritical-all-B-finite}
We begin by showing assertions   (1.i) and (2.i).
Let $\kappa >0$. By Lemma~\ref{L:Holder-estimates},
\begin{equation} \label{apr17}
	\int_{\rn}
		e^{\left(\kappa |u| \right)^{\frac{2\beta}{2+\beta}}}\,\dgn
		\le \sqrt{\frac{2}{\pi}}
			\int_{0}^{\infty} e^{\left[ \kappa\FF_{m^B}(t) \right]^{\frac{2\beta}{2+\beta}}
				-\frac{t^2}{2}}\,\d t
\end{equation}
for any weakly differentiable function $u$ obeying
\eqref{E:nabla-Marcinkiewicz-m} and \eqref{mu}. Here,
$\FF_{m^B}$ is the function  given by \eqref{E:F-mB}.   Next, Lemma~\ref{L:B-integral} tells us that
\begin{equation*}
	\bigl[\kappa
		\FF_{m^B}(t)
	\bigr]^\frac{2\beta}{2+\beta}
		= \left( \frac{\kappa}{\kappab} \right)^\frac{2\beta}{2+\beta}
			\frac{t^2}{2}
			+ \cdots
		\quad\text{as $t\to\infty$}.
\end{equation*}
Hence, the integral on the right-hand side of \eqref{apr17} converges whenever
$\kappa<\kappab$. This proves  properties (1.i) and (2.i) under condition
\eqref{E:nabla-Marcinkiewicz-m}.  If $u$ satisfies condition
\eqref{E:nabla-Marcinkiewicz-M}, then \eqref{E:nabla-Marcinkiewicz-m} also
holds just owing to  \eqref{AtoM}. Hence, properties (1.i) and (2.i) follow
also in this case.

Assertion (1.ii) for the $M^B$ norm is a straightforward consequence of
Proposition~\ref{P:MB-critical-beta-small}, and for the $m^B$ quasi-norm it
requires  the additional use of inequality \eqref{AtoM}.

Property (2.ii) follows via Proposition~\ref{P:LB-critical-beta-large} and
Proposition~\ref{P:mB-critical-beta-large}, combined  with  inequalities
\eqref{AtoM}.

Finally, assertion (2.iii) is treated in
Proposition~\ref{P:LB-supercritical-beta-all}. Inequality \eqref{AtoM} has to
be exploited here as well.
\end{proof}

\begin{proof}[Proof of Theorem~\ref{T:L-infty}]
Assume that the function $u\in W^1L^\infty\RG$ fulfills conditions \eqref{mu}
and \eqref{Linfinity}. By Lemma~\ref{L:Holder-estimates-Linfty},
\begin{equation} \label{jun5}
	\int_{\rn} e^{(\kappa u)^2}\,\dgn
		\le \sqrt{\frac{2}{\pi}}
			\int_{0}^\infty e^{\left[\kappa \FF_{L^\infty}(t)\right]^2 - \frac{t^2}{2}}\,\d t,
\end{equation}
where the function $\FF_{L^\infty}$ is given by \eqref{E:F-infty}. By equation \eqref{E:Phi-prime},
\begin{equation*}
	-\Phi'(t)-t\Phi(t)\to 0
		\quad\text{as $t\to\infty$}.
\end{equation*}
Consequently,
\begin{equation}\label{july100}
	\left[ \kappa\FF_{L^\infty}(t) \right]^2 - \frac {t^2}{2}
		= \left( \kappa^2-\frac{1}{2} \right)t^2  + \cdots
		\quad\text{as $t\to\infty$,}
\end{equation}
under either  assumption $\med(u)=0$, or $\mv(u)=0$.
Thanks to equation \eqref{july100}, the integral on the right-hand side of
inequality \eqref{jun5} converges for every $\kappa\in(0,\tfrac1{\sqrt{2}})$.
Part (i) of the statement is thus established.

In order to show part (ii), it clearly suffices to assume that $\kappa =
\frac 1{\sqrt{2}}$. Consider the function $u\colon\rn\to\R$ defined as
$u(x)=x_1$ for $x\in\rn$. Trivially, $u\in W^1L^\infty\RG$, and  since
$|\nabla u(x)|=1$ for every $x\in\rn$, the function $u$ fulfills assumption
\eqref{Linfinity}.  Moreover, $\med(u)=\mv(u)=0$, and therefore condition
\eqref{mu} is fulfilled as well in both its variants. Notice that
$u^\circ(s)=\Phi^{-1}(s)$ for $s\in(0,1)$. Thereby,
\begin{align*}
	\int_{\rn} e^{\frac{1}{2}u^2}\,\dgn
		= \int_{0}^1 e^{\frac{1}{2}u^\circ(s)^2}\,\d s
		= \int_{0}^1 e^{\frac{1}{2}\Phi^{-1}(s)^2}\,\d s
		= \frac{1}{\sqrt{2\pi}}
				\int_{-\infty}^\infty e^{\frac{\tau^2}{2}} e^{-\frac{\tau^2}{2}}\,\d \tau
		= \infty,
\end{align*}
where the last but one equality holds by the  change of variables $s=\Phi(\tau)$.
The proof  is  complete.
\end{proof}

\begin{proof}[Proof of Theorem~\ref{T:E-space}]
Let $u\in W^1 \exp E^{\beta}\RG$ and let $B$ be a Young function satisfying
\eqref{BN}.  We may assume, without loss of generality, that
\begin{equation*}
	\|\nabla u\|_{L^B\RG}
		\le 1.
\end{equation*}
Let us also assume, for the time being, that $\med (u)=0$.  Let
$s_0\in(0,\tfrac 12)$. By the local absolute continuity of  $u^\circ$ and
H\"older's inequality \eqref{E:Holder},
\begin{align*}
	u^\circ(s)
		& = \int_{s}^{\frac12} -{u^\circ}'(r)\,\d r
			 = \int_{s}^{s_0} -{u^\circ}'(r)I(r) \frac{\,\d r}{I(r)}
			+ \int_{s_0}^{\frac12} -{u^\circ}'(r)I(r) \frac{\,\d r}{I(r)}
			\\
		& \le \lVert -{u^\circ}' I \rVert_{L^B(0,s_0)}
				\opnorm*{\frac{1}{I}}_{L^{\tilde B}(s,s_0)}
                 	+ \lVert -{u^\circ}' I \rVert_{L^B(0,\frac12)}
				\opnorm*{\frac{1}{I}}_{L^{\tilde B}(s_0,\frac12)}
			\quad\text{for $s\in (0,s_0)$.}
\end{align*}
Observe that
\begin{equation} \label{nov2-a}
	\lVert -{u^\circ}' I\rVert_{L^B(0,s_0)}
		\le \|(-{u^\circ}'I)^*\|_{L^B(0,s_0)}
		\le \| |\nabla u|^*\|_{L^B(0,s_0)}
		= \| \nabla u \|_{L^B(E)}
\end{equation}
for a suitable measurable set $E\subset\rn$ such that $\gamma_n(E)=s_0$.  Notice that
the first inequality in \eqref{nov2-a} holds as a consequence of the
Hardy-Littlewood principle \citep[Chapter~2, Theorem~4.6]{Ben:88},
since, owing to inequality \eqref{HL},
\begin{equation*}
	\int_0^s  (-{u^\circ}' I\chi_{(0,s_0)})^*(r)\, \d r
		\le \int_0^s (-{u^\circ}' I)^*\chi_{(0,s_0)}(r)\, \d r
			\quad \text{for $s\in(0,1)$.}
\end{equation*}
On the other hand, the second inequality in \eqref{nov2-a} is a consequence of
inequality \eqref{PS}.  Fix $\varrho>0$.  Since $|\nabla u|$ has an absolutely
continuous norm, by \eqref{nov2-a} one can choose $s_0$  so small that
\begin{equation*}
	\lVert -{u^\circ}' I \rVert_{L^B(0,s_0)}
		\le \| \nabla u \|_{L^B(E)} < \varrho.
\end{equation*}
Next, owing to Proposition~\ref{P:PSnorm},
\begin{equation*}
	\lVert -{u^\circ}' I \rVert_{L^B(0,\frac12)}
		\le \|\nabla u\|_{L^B\RG}
		\le 1,
\end{equation*}
and
\begin{equation*}
	\opnorm*{\frac{1}{I}}_{L^{\tilde B}(s_0,\frac12)}
		\le \frac{1}{I(s_0)} \opnorm*{1}_{L^{\tilde B}(0,\frac12)}
		= \frac{c}{I(s_0)},
\end{equation*}
where we have set $c =  \opnorm*{1}_{L^{\tilde B}(0,\frac12)}$.
Thanks to Lemma~\ref{L:I-norm-sharp} and equation \eqref{E:Phi-invers},
there exists a constant $C$ such that
\begin{equation*}
	\opnorm*{\frac{1}{I}}_{L^{\tilde B}(s,s_0)}
		\le C \Bigl( \log\frac{1}{s} \Bigr)^\frac{2+\beta}{2\beta}
		\quad\text{for $s\in(0,s_0)$}.
\end{equation*}
Altogether, we obtain that
\begin{align}\label{june40'}
	 0\leq  \,u^\circ(s)
		& \le C\varrho \Bigl( \log\frac{1}{s} \Bigr)^\frac{2+\beta}{2\beta}
		 	+ \frac{c}{I(s_0)}
		\quad\text{for $s\in(0,s_0)$}.
\end{align}
If the assumption that $\med (u)=0$ is dropped, from an application of
inequality \eqref{june40'} to the function $u- \med (u)$ one infers that
\begin{equation*}
	|u^\circ(s)|
		 \le C\varrho \Bigl( \log\frac{1}{s} \Bigr)^\frac{2+\beta}{2\beta}
		 	+ \frac{c}{I(s_0)} + |\med (u)|
		\quad\text{for $s\in(0,s_0)$}.
\end{equation*}
Now,
\begin{align} \label{nov0}
	\int_{\rn} e^{(\kappa|u|)^\frac{2\beta}{2+\beta}}\,\dgn
		& = \int_{0}^{1} e^{(\kappa |u^\circ(s)|)^\frac{2\beta}{2+\beta}}\,\d s
			=  \int_{0}^{\frac12} e^{(\kappa |u^\circ(s)|)^\frac{2\beta}{2+\beta}}\,\d s
				+ \int_{\frac12}^1 e^{(\kappa |u^\circ(s)|)^\frac{2\beta}{2+\beta}}\,\d s
\end{align}
and
\begin{equation}\label{july20}
	\int_{0}^{\frac12} e^{(\kappa |u^\circ(s)|)^\frac{2\beta}{2+\beta}}\,\d s
		= \int_{0}^{s_0} e^{(\kappa |u^\circ(s)|)^\frac{2\beta}{2+\beta}}\,\d s
			+ \int_{s_0}^{\frac12} e^{(\kappa |u^\circ(s)|)^\frac{2\beta}{2+\beta}}\,\d s.
\end{equation}
The second  integral on the right-hand side of equation \eqref{july20} is
finite, since $u^\circ$ is bounded in $(s_0, \frac12)$.  As for the first  one,
\begin{align*}
	\int_{0}^{s_0} e^{(\kappa |u^\circ(s)|)^\frac{2\beta}{2+\beta}}\,\d s
		& \le \int_{0}^{s_0} e^{
			\left(
				\kappa C\varrho \bigl( \log\frac{1}{s} \bigr)^\frac{2+\beta}{2\beta}
				+ \kappa\frac{c}{I(s_0)} + \kappa |\med (u)|
			\right)^\frac{2\beta}{2+\beta}}\,\d s
			\\
		& \le \int_{0}^{s_0} e^{C_1
			\left(
				\varrho^\frac{2\beta}{2+\beta} \log\frac{1}{s}
				+ \bigl(\frac{c}{I(s_0)}+|\med (u)|\bigr)^\frac{2\beta}{2+\beta}
			\right)}\,\d s
			\\
		& = e^{C_1 \bigl(\frac{c}{I(s_0)}+|\med (u)|\bigr)^\frac{2\beta}{2+\beta}}
			\int_{0}^{s_0}
				\Bigl( \frac{1}{s} \Bigr)^{C_1\varrho^\frac{2\beta}{2+\beta} }
				\d s
\end{align*}
for some constant $C_1$. Clearly, the last integral is convergent,
provided $\varrho$ is chosen small enough. Thus,
\begin{align} \label{nov4}
	\int_{0}^{\frac12} e^{(\kappa |u^\circ(s)|)^\frac{2\beta}{2+\beta}}\,\d s
		< \infty.
\end{align}
An analogous argument, exploiting equation \eqref{symmetry}, shows that
\begin{align} \label{nov5}
	\int_{\frac12}^1 e^{(\kappa |u^\circ(s)|)^\frac{2\beta}{2+\beta}}\,\d s
		< \infty.
\end{align}
Property \eqref{E:E-space-kappa} follows via \eqref{nov0}, \eqref{nov4} and
\eqref{nov5}.
\end{proof}

\paragraph{Acknowledgment}
We wish to thank Debabrata Karmakar for pointing out several misprints in a
preliminary version of our paper. We also thank the referee for his careful
reading of the paper and for his valuable comments.

\section*{Compliance with Ethical Standards}

\subsection*{Funding}

This research was partly funded by:

\begin{enumerate}
\item Research Project 2015HY8JCC of the Italian Ministry of University and
Research (MIUR) Prin 2015 ``Partial differential equations and related
analytic-geometric inequalities'';
\item GNAMPA of the Italian INdAM -- National Institute of High Mathematics
(grant number not available);
\item Grant P201-18-00580S of the Czech Science Foundation;
\item Grant SVV-2017-260455 of the Charles University.
\end{enumerate}

\subsection*{Conflict of Interest}

The authors declare that they have no conflict of interest.


\end{document}